\documentclass[]{amsart}

\usepackage{amsmath}
\usepackage{amssymb,amsfonts}
\usepackage{xypic}
\xyoption{all}

\newtheorem{theorem}{Theorem}[section]
\newtheorem{definition}[theorem]{Definition}
\newtheorem{lemma}[theorem]{Lemma}
\newtheorem{proposition}[theorem]{Proposition}

\newtheorem{corollary}[theorem]{Corollary}

\newtheorem{fact}[theorem]{Fact}

\newcommand{\Ob}{\mathrm{Ob}}

\newcommand{\C}{\mathcal{C}}
\newcommand{\D}{\mathcal{D}}
\newcommand{\E}{\mathcal{E}}

\newcommand{\Set}{\mathbf{Set}}
\newcommand{\SSet}{\mathbf{SSet}}

\newcommand{\Cat}{\mathbf{Cat}}

\newcommand{\Diag}{\mathrm{Diag}}

\newcommand{\tb}{|\!|}
\newcommand{\thb}{|\!|\!|}
\newcommand{\ner}{{\mathrm{N}}}
\newcommand{\W}{{\overline{W}}}

\begin{document}
\title{Homotopy colimits of 2-functors}

\author{A.M. Cegarra, B.A.  Heredia}
\thanks{This work has been supported by DGI
of Spain, Project MTM2011-22554. Also, the second author by FPU
grant AP2010-3521.}

\address{\newline
Departamento de \'Algebra, Facultad de Ciencias, Universidad de
Granada.
\newline 18071 Granada, Spain \newline
 acegarra@ugr.es\ baheredia@ugr.es }

 \subjclass[2000]{18D05, 55P15, 18F25}

\keywords{2-categories, classifying space, homotopy colimit}

\begin{abstract} Like categories, small 2-categories have
well-understood classifying spaces. In this paper, we deal with
homotopy types represented by 2-diagrams of 2-categories. Our
results extend to homotopy colimits of 2-functors lower categorical
analogues that have been classically used in algebraic topology and
algebraic K-theory, such as the Homotopy Invariance Theorem (by
Bousfield and Kan), the Homotopy Colimit Theorem (Thomason),
Theorems A and B (Quillen), or the Homotopy Cofinality Theorem
(Hirschhorn).

\vspace{0.3cm}
\begin{center}{\em Dedicated to Ronald Brown on his 80th
birthday}\end{center}
\end{abstract}

\maketitle
\section{Introduction}
A 2-category is a category $C$ in which the morphism sets $C(c,c')$
are categories and the composition  is a functor. Every small
2-category $C$ has associated a simplicial category  $\ner
C:\Delta^{\mathrm{op}}\to \Cat$, called its nerve, with category of
$p$-simplices
$$
\coprod_{c_0,\ldots,c_p}\hspace{-0.1cm}
C(c_0,c_1)\times C(c_1,c_2)\times\cdots\times C(c_{p-1},c_p),
$$
whose Segal geometric realization is, by definition, the classifying
space of the 2-category. Recently, there has been rising interest in
the relation between 2-categories and the homotopy types of their
classifying spaces; see, for example,  the papers by  Baas,
B\"{o}kstedt and Kro ~\cite{BBT2012}, Bullejos and Cegarra
~\cite{B-C2003,Cegarra-2011}, Chiche ~\cite{Chich2014,Chich2015,
Chich2015-2}, del Hoyo ~\cite{delHoyo2012} and, particularly, the
paper by Ara and Maltsiniotis ~\cite{AM2014}, where it is fully
proved that the category $\mathbf{2}\Cat$, of small 2-categories and
2-functors, has a Thomason model structure (as  first
announced by Worytkiewicz, Hess, Parent and Tonks in
~\cite{WHPT2007}) such that the classifying space functor is an
equivalence of homotopy theories between 2-categories and
topological spaces.

This paper focuses on the study of homotopy colimits of 2-functors
$\D:C\to \underline{\mathbf{2}\Cat}$, from an indexing 2-category
$C$ into the 2-category $\underline{\mathbf{2}\Cat}$ of
2-categories, 2-functors, and 2-natural transformations, and our
results extend lower categorical analogues
classically used in algebraic topology and algebraic K-theory. We
shall stress that the main difference with the ordinary case of
functors $C\to \Cat$, where $C$ is a category, is that now there are
2-cells $\phi:p\Rightarrow p'$ in $C$ that produce  2-natural
transformations $\D\phi:\D p\Rightarrow \D p'$, and therefore
homotopies between the induced maps on classifying spaces by the
associated 2-functors $\D p$ and $\D p'$, which must be taken into
account. The fundamental ``homotopy colimit" construction we deal
with
 functorially assigns to each 2-functor
$\D:C\to \underline{\mathbf{2}\Cat}$ a simplicial 2-category
$\mathrm{hocolim}_C\D:\Delta^{\mathrm{op}}\to \mathbf{2}\Cat$, whose
category of $p$-simplices is
\begin{equation}\label{int-2}
\coprod_{c_0,\ldots,c_p}\hspace{-0.1cm} \D_{c_0}\!\times
C(c_0,c_1)\times C(c_1,c_2)\times\cdots\times C(c_{p-1},c_p).
\end{equation}
In the case where $C$ is a category and $\D:C\to
 \Cat$ is a functor into the category of categories and functors, one gets
 the usual definition of homotopy
 colimit of a diagram of (nerves of) categories by Bousfield and Kan
 ~\cite[Chp. XII]{Busf-Kan1972}. Also, for $C$ a 2-category
 and $\D:C\to \underline{\Cat}\subseteq \underline{\mathbf{2}\Cat}$ a 2-functor
 into the category of categories, functors, and
 natural transformations, one recovers the homotopy colimit
 of a 2-diagram of categories by Hinich and Schechtman
 in ~\cite{Hinich1985}.

Interesting 2-diagrams of 2-categories naturally arise from basic
problems in  homotopy theory of 2-categories. For example, the
analysis of the homotopy fibres of the map  induced on
classifying spaces by a 2-functor $F:A\to C$ leads to the study of the
2-functor $F\!\downarrow\!-:C\to \underline{\mathbf{2}\Cat}$, which
associates to each object $c\in C$ the slice 2-category
$F\!\downarrow\!c$ ~\cite{Gray1969}, whose objects are 1-cells in $C$
of the form $p:Fa\to c$. A basic observation here  is that the
forgetful 2-functors $F\!\downarrow \!c\to A$ assemble to define a
simplicial weak equivalence $\xymatrix@C=16pt{\mathrm{hocolim}_C(F\!
\downarrow\!-)\ar@{}@<-2pt>[r]^(.75){\sim}\ar[r]&A}$, so that
$\mathrm{hocolim}_C(F\! \downarrow\!-)$ can be regarded as a
simplicial resolution of the `total' 2-category of the 2-functor
$F$. See ~\cite[Theorem 3.2]{Cegarra-2011}, where Quillen's Theorem B
is generalized for 2-functors, or ~\cite[Th\'eor\`{e}m
2.34]{Chich2015}, where a relative Quillen's Theorem A for
2-functors is given.

There is another source for 2-diagrams of 2-categories: The study
and classification of cofibred 2-categories. The well-known
Grothendieck correspondence between covariant pseudo-functors and
cofibred categories ~\cite{Grothendieck1971} has been generalized to
bicategories by Bakovi\'c ~\cite{Bakovic2011} and Buckley
~\cite{Bakovic2011}. The latter authors, in particular, prove that there is a
``Grothendieck construction" on 2-functors $\D:C\to
\underline{\mathbf{2}\Cat}$ that gives rise to 2-categories
$\int_C\D$ endowed with a split 2-cofibration over $C$, and this
correspondence $\D\mapsto \int_C\D$ is the function on objects of an
equivalence between the 3-category of 2-functors $\D:C\to
\underline{\mathbf{2}\Cat}$ and the 3-category of split 2-cofibred
2-categories over $C$. With Thomason's Homotopy Colimit Theorem as
its natural precedent, a main result of our paper shows that, for
any 2-functor $\D:C\to \underline{\mathbf{2}\Cat}$, the geometric
realization of the simplicial category $\mathrm{hocolim}_C\D$ has
the homotopy type of the classifying space of the 2-category
$\int_C\D$.

The plan of the paper is as follows. After this introductory
section, the paper is organized into eight sections. Section 2
comprises some notations and a brief review of basic facts
concerning classifying spaces of 2-categories that we are going to
use later. In Section 3, we present the homotopy colimit
construction on 2-functors $\D:C\to \underline{\mathbf{2}\Cat}$ and,
mainly,  show its homotopy invariance property: If $\D\to \E$ is a
2-transformation that is locally a weak equivalence, then the
induced $\mathrm{hocolim}_C\D\to \mathrm{hocolim}_C\E$ is also  a
weak equivalence. In  Section 4, we briefly review the
Grothendieck construction on 2-diagrams of 2-categories. This
construction will be extensively used throughout the paper and
especially in  Section 5, which is fully dedicated to proving
the aforementioned extension of Thomason's Homotopy Colimit Theorem
for 2-diagrams of 2-categories. This is quite a long and technical
section, but crucial for our conclusions. After this, both
 constructions, $\mathrm{hocolim}_C$ and $\int_C-$,  can be interchanged
for homotopy purposes.
 In Section 6, we review basic facts concerning the homotopy-fibre 2-functors
$F\!\downarrow\!-:C\to \underline{\mathbf{2}\Cat}$ associated to
2-functors $F:A\to C$, which are needed in  Sections 7
and 8. In Section 7, we deal with questions such as: when does a
2-transformation $\Gamma:\D\Rightarrow\E$, between 2-functors
$\D,\E:C\to \underline{\mathbf{2}\Cat}$, induce a homotopy left
cofinal 2-functor $\int_C\Gamma:\int_C\D\to\int_C\E$? Or when are
the canonical squares ($c\in\Ob C$, $y\in\Ob \E_c$)
$$\xymatrix{\Gamma_{\!c}\!\downarrow\!y\ar[r]\ar[d]
& \E_c\!\downarrow\!y\ar[d]\\
\int_C\D\ar[r]^{\int_C\Gamma} & \int_C\E}$$  homotopy pullbacks?
Our main results here are actually extensions of the well-known
Quillen's Theorems A and B for functors between categories to
morphisms between 2-diagrams of 2-categories. The final Section 8 is
dedicated to analyzing the behavior of the homotopy colimit
construction when a 2-functor $\D:C\to \underline{\mathbf{2}\Cat}$
is composed with a 2-functor $F:A\to C$. There is a canonical
2-functor $\int_AF^*\D\to \int_C\D$, and we mainly study when this
2-functor is a weak equivalence or, more interestingly, when the
canonical pullback square in $\mathbf{2}\Cat$
$$
\xymatrix{\int_AF^*\D\ar[r]\ar[d]&\int_C\D\ar[d]\\
A\ar[r]^{F}&C }
$$
is a homotopy pullback.

\section{Classifying spaces of 2-categories}\label{1.1}
In Quillen's development of $K$-theory ~\cite{Quillen1973}, the
higher $K$-groups are defined as the homotopy groups of a
topological {\em classifying space\footnote{For `space' we mean a
compactly generated Hausdorff space, and $\mathbf{Top}$ is the
category of these spaces.}} $|C|$ associated to a (small) category
$C$. This space is a CW-complex defined as $$|C|=|\ner C|,$$ the
Milnor geometric realization of the simplicial set termed its
Grothendieck nerve
\begin{equation}\label{GrothendieckNerve}
\ner C: \Delta^{\mathrm{op}}\to \Set,
\ \ [p]\mapsto  \ner_pC=\coprod_{c_{0},\ldots,c_p}\hspace{-0.2cm}
  C(c_{0},c_{1})\times C(c_{1},c_{2})\times\cdots\times C(c_{p-1},c_p),
\end{equation}
 whose $p$-simplices are length $p$
sequences of composable morphisms in $C$ ($\ner_0 C=\Ob C$).

In ~\cite{Segal1974}, Segal extended Milnor's geometric realization
process to simplicial (compactly generated topological) spaces. If
$S:\Delta^{\!{\mathrm{op}}}\to \Cat$ is a simplicial category, by
replacing each category $S_p$ by its classifying space $|S_p|$, one
obtains a simplicial space, $|S|:\Delta^{\!{\mathrm{op}}}\to
\mathbf{Top}$,  whose Segal realization is the {\em classifying
space $\tb S\tb$ of the simplicial category $S$}. By ~~\cite[Lemma in
page 86]{Quillen1973}, there is a natural homeomorphism
$$
\tb S \tb =|[p]\mapsto |S_p||=|[p]\mapsto |[q]\mapsto \ner_qS_p||\cong |\Diag \ner S|,
$$
where $\Diag\ner S:\Delta^{\mathrm{op}}\to \Set$, $[p]\mapsto
\ner_pS_p$, is the simplicial set diagonal of the bisimplicial set
$\ner S:\Delta^{\!{\mathrm{op}}}\times \Delta^{\!{\mathrm{op}}}\to
\Set$, $([p],[q])\mapsto \ner_qS_p $, obtained by composing
$S:\Delta^{\!{\mathrm{op}}}\to \Cat$ with the nerve functor $\ner:
\Cat \to \SSet$ from categories to simplicial sets.

The notion of classifying space of a simplicial category provides
the usual definition of the classifying space of a 2-category. Although
for the general background on 2-categories used in this paper we
refer to ~\cite{borceux1994} and ~\cite{Street1996},  to fix some
notation and terminology, we shall recall that a 2-{\em category}
$C$ is just a category enriched in the category of small categories.
Then, $C$ is a category in which the hom-set between any two objects
$c,c'\in C$ is the set of objects of a category $C(c,c')$, whose
objects $p:c\to c'$ are called 1-{\em cells} and whose arrows are
called {\em $2$-cells} and are denoted by $\alpha:p\Rightarrow p'$
and depicted as
$$\xymatrix @C=8pt {c \ar@/^0.7pc/[rr]^{p} \ar@/_0.7pc/[rr]_{p'}
\ar@{}[rr]|{\Downarrow\!\alpha } & &c'.}$$ Composition in each
category $C(c,c')$, usually referred to as the vertical composition
of 2-cells, is denoted by $\alpha\cdot \beta$. Moreover, the
horizontal composition is a functor
$$\C(c,c')\times\C(c',c'')\overset{\circ\ }\rightarrow \C(c,c'')\hspace{0.5cm}((x,y)\mapsto y\circ x)$$
that is associative and has identities $1_c\in\C(c,c)$.

 For any
2-category $C$,  the nerve construction \eqref{GrothendieckNerve} on
it actually works by giving a simplicial category $\ner C:
\Delta^{\mathrm{op}}\to \Cat$, whose Segal's classifying space is
then the {\em classifying space $\tb C\tb$ of the $2$-category}.
Thus,
$$
\tb C\tb =|\Diag\ner\ner C |,
$$
where $\ner\ner C:\Delta^{\!{\mathrm{op}}}\times
\Delta^{\!{\mathrm{op}}}\to \Set $ is the {\em double nerve} of $C$,
$([p],[q])\mapsto \ner_q\ner_p C$.

Like  $\Cat$, the category  $\mathbf{2}\Cat$, of small 2-categories
and 2-functors, has a Thomason model structure, as first
announced by Worytkiewicz, Hess, Parent and Tonks in
~~\cite{WHPT2007} and fully proved by Ara and Maltsiniotis in
~~\cite[Th\'eo\`{e}me 6.27]{AM2014}, such that the classifying space
functor $C \mapsto \tb C\tb$ is an equivalence of homotopy theories
between
 2-categories and topological spaces. Thus, for example,
a 2-functor $F:A\to C$ is a {\em weak equivalence if and only if the
induced map $\tb F\tb:\tb A\tb\to \tb C\tb$ is a homotopy
equivalence}, and a commutative square of $2$-categories and
$2$-functors
$$
\xymatrix@C=20pt@R=20pt{P\ar[r]\ar[d]&B\ar[d]\\ A\ar[r]&C
 }
$$
{\em is a homotopy pullback if and only if the induced square on
classifying spaces
$$
\xymatrix@C=20pt@R=20pt{\tb P\tb \ar[r]\ar[d]&\tb B\tb \ar[d]\\ \tb A\tb \ar[r]&\tb C\tb
 }
$$
is a homotopy pullback of spaces}. Later, we shall use basic
properties of homotopy pullback squares of spaces, such as the {\em
two out of three property}, etc. (see ~\cite[\S 5]{JChPJ2006} for
instance). In particular, the {\em homotopy-fibre characterization},
which easily leads us to assert that the square of 2-categories
above is a homotopy pullback whenever, for any object $a\in A$,
there is a commutative diagram of 2-categories and 2-functors
$$
\xymatrix{ P_a\ar[r]\ar[d]&P\ar[d]\ar[r]&B\ar[d]\\
A_a\ar[r]^{F_a}&A\ar[r]&C
}
$$
such that  $a\in \mathrm{Im}F_a$,  the space $\tb A_a\tb$ is
contractible, and both the left square and the composite square are
homotopy pullbacks.

We shall call a 2-category $C$ {\em weakly
contractible}\footnote{These are called `aspherical'  by Cisinski in
~\cite{Cisinski2006} and Chiche in ~\cite{Chich2015}.} whenever the
functor from $C$ to the terminal (only one 2-cell) 2-category $C\to
\mathrm{pt}$ is a weak equivalence, that is, if the classifying
space $\tb C\tb$ is contractible.

The following fact will be also used.

\begin{fact}[~\cite{Cegarra-2011} Lemma 2.6]\label{fact1} If two $2$-functors
between $2$-categories $F,G:A\to C$ are related by a lax or oplax
transformation, $F\Rightarrow G$, then there is an induced homotopy,
$\tb F\tb \Rightarrow \tb G\tb$, between the induced maps on
classifying spaces  $\tb F\tb,\tb G\tb :\tb A\tb \to \tb C \tb$.
\end{fact}

To conclude this preliminary section, we shall recall that the {\em
classifying space $\thb S\thb$ of a simplicial $2$-category}
$S:\Delta^{\!{\mathrm{op}}}\to \mathbf{2}\Cat$, is the geometric
realization of the simplicial space obtained by composing $S$ with
the classifying space functor $\tb\,\text{-}\,\tb:\mathbf{2}\Cat\to
\mathbf{Top}$. Therefore,
$$
\thb S\thb =|[p]\mapsto \tb S_p\tb|=|[p]\mapsto |\Diag\ner\ner S_p| |=|\Diag \ner\ner S|,
$$
where $\Diag \ner\ner S$ is the simplicial set, $[p]\mapsto
\ner_p\ner_pS_p$, diagonal of the trisimplicial set $\ner\ner S$.

\section{The homotopy colimit construction on 2-functors}
We shall start by fixing some notations. Throughout the paper, the
2-category of (small) 2-categories, 2-functors, and 2-natural
transformations is denoted by $\underline{\mathbf{2}\Cat}$ (whereas
$\mathbf{2Cat}$, recall, denotes its underlying category of
2-categories and 2-functors). We view any category as a 2-category
in which all its 2-cells are identities, and thus
 $\underline{\Cat} \subseteq \underline{\mathbf{2}\Cat}$ is the 2-subcategory consisting of
 categories, functors, and natural transformations.

Further, if $C$ is a 2-category, the effect on cells of any
2-functor $\D:C \to \underline{\mathbf{2}\Cat}$ is denoted by
\begin{equation}\label{efcells}\xymatrix@C=0.5pc{c \ar@/^0.7pc/[rr]^{ f} \ar@/_0.7pc/[rr]_{
g}\ar@{}[rr]|{\Downarrow\alpha} &
 &c' }\mapsto \xymatrix@C=0.6pc{\D_{\!c}  \ar@/^0.8pc/[rr]^{ f_*}
\ar@/_0.8pc/[rr]_{ g_*}\ar@{}[rr]|(.55){\Downarrow\alpha_*} &
&\D_{\!c'}}.
\end{equation}
or, if $\D:C^{\mathrm{op}} \to \underline{\mathbf{2}\Cat}$ is
contravariant, by
$$\xymatrix@C=0.5pc{c \ar@/^0.7pc/[rr]^{ f} \ar@/_0.7pc/[rr]_{
g}\ar@{}[rr]|{\Downarrow\alpha} &
 &c' }\mapsto \xymatrix@C=0.6pc{\D_{\!c'}  \ar@/^0.8pc/[rr]^{ f^*}
\ar@/_0.8pc/[rr]_{ g^*}\ar@{}[rr]|(.55){\Downarrow\alpha^*} &
&\D_{\!c}}.
$$

The  notion of classifying space for simplicial 2-categories
naturally leads to the notion of {\em classifying space for
$2$-diagrams of $2$-categories}, that is, for 2-functors from a
2-category to $\underline{\mathbf{2}\Cat}$, which is given through
the {\em homotopy colimit} or {\em Borel} construction
 as shown below.

\begin{definition} Let $C$ be a $2$-category and $\D:C\to
\underline{\mathbf{2}\Cat}$ be a $2$-functor. The homotopy colimit
of $\D$ is the  simplicial $2$-category {\em
\begin{equation}\label{hocol}\mbox{hocolim}_C\D:\Delta^{\!{\mathrm{op}}}\to
\mathbf{2}\Cat,
\end{equation}
} whose $2$-category of $p$-simplices is
$$
 \coprod_{c_0,\ldots,c_p}\hspace{-0.1cm} \D_{\!c_0}\!\times
C(c_0,c_1)\times C(c_1,c_2)\times\cdots\times C(c_{p-1},c_p),
$$
and whose face and degeneracy $2$-functors are defined as follows:
The face $2$-functor $d_0$ is induced by the $2$-functor
$$
d_0: \D_{\!c_0}\times C(c_0,c_1)\to \D_{\!c_1},
$$
which carries an object $(x,c_0\overset{f}\to c_1)$ to the object
$f_*x$. A $1$-cell $(u,\alpha):(x,f)\to (y,g)$ is carried by $d_0$
to the composite $1$-cell $g_*u\circ \alpha_*x:f_*x\to g_*y$, $$
\xymatrix{f_*x\ar[r]^{\alpha_*x}&g_*x\ar[r]^{g_*u}&g_*y},
$$
 and $d_0$ acts on $2$-cells by
$$
\xymatrix@C=30pt{(x,f) \ar@/^0.8pc/[r]^{(u,\alpha)}
\ar@/_0.8pc/[r]_{(v,\alpha)}\ar@{}[r]|{\Downarrow\!(\phi,1_\alpha)}
&(y,g)}\ \mapsto \xymatrix@C=50pt{f_*x \ar@/^0.8pc/[r]^{g_*u\circ
\alpha_*x} \ar@/_0.8pc/[r]_{g_*v\circ
\alpha_*x}\ar@{}[r]|{\Downarrow \,g_*\!\phi\circ 1_{\!\alpha_*\!x}}
&g_*y.}
$$
The other face and degeneracy $2$-functors  are induced by
the operators $d_i$ and $s_i$ in $\ner C$ as $1_{\!\D_{\!c_0}}\!\times d_i$ and $1_{\!\D_{\!c_0}}\!\times s_i$,
respectively.

The classifying space of the $2$-functor $\D$ is {\em $\thb
\text{hocolim}_C\D \thb,$ } the classifying space of its homotopy
colimit simplicial $2$-category.
\end{definition}

In the case where $C$ is a category and $\D:C\to
 \Cat$ is a functor, one gets the usual definition for a homotopy
 colimit of a diagram of (nerves of) categories by Bousfield and Kan ~\cite[Chp. XII]{Busf-Kan1972}. Also, for $C$ a 2-category
 and $\D:C\to \underline{\Cat}$ a 2-functor, one recovers the homotopy colimit
 construction of a 2-diagram of categories by Hinich and Schechtman
 in ~\cite[Definition (2.2.2)]{Hinich1985}.

It is not hard to see that, if $\D,\E:C\to
\underline{\mathbf{2}\Cat}$ are 2-functors, then any
2-transformation $\Gamma:\D\Rightarrow \E$ gives rise to a
simplicial functor $\Gamma_*:\mathrm{hocolim}_C\D\to
\mathrm{hocolim}_C\E$; and also that a modification $m:\Gamma
\Rrightarrow \Gamma'$, where $\Gamma':\D\Rightarrow\E$ is any other
2-transformation, induces a simplicial transformation
$m_*:\Gamma_*\Rightarrow \Gamma'_*$. Thus, the homotopy colimit
construction provides a 2-functor
$$
\mathrm{hocolim}_C-:\underline{\mathbf{2}\Cat}^C \to
\underline{\mathbf{2}\Cat}^{\Delta^{\mathrm{op}}}.
$$

Similarly, if $C$ is a $2$-category and $\D:C^{\mathrm{op}}\to
\underline{\mathbf{2}\Cat}$ is any $2$-functor, we call the {\em
homotopy colimit of $\D$} the simplicial $2$-category\footnote{Note
that $\mathrm{hocolim}_C\D\neq \mathrm{hocolim}_{C^{\mathrm{op}}}\D$.}
\begin{equation}\label{hocolop}\mathrm{hocolim}_C\D:\Delta^{\!{\mathrm{op}}}\to
\mathbf{2}\Cat
\end{equation}
 whose $2$-category of $p$-simplices is
$$
 \coprod_{c_0,\ldots,c_p}\hspace{-0.1cm}
C(c_0,c_1)\times C(c_1,c_2)\times\cdots\times C(c_{p-1},c_p)\times\D_{\!c_p}
$$
and whose faces and degeneracies are  induced by the corresponding ones
in $\ner C$, as $d_i\times 1_{\!\D_{\!c_p}}$ and $s_i\times
1_{\!\D_{\!c_p}}$, for $0\leq i<p$,  whereas the face $2$-functor
$d_p$ is induced by the $2$-functor
$$
d_p: C(c_{p-1},c_p)\times \D_{\!c_p}\to \D_{\!c_{p-1}},
$$
which acts on cells by
$$
\xymatrix@C=30pt{(x,f) \ar@/^0.8pc/[r]^{(u,\alpha)}
\ar@/_0.8pc/[r]_{(v,\alpha)}\ar@{}[r]|{\Downarrow\!(\phi,1_\alpha)}
&(y,g)}\ \mapsto \xymatrix@C=50pt{f^*x \ar@/^0.8pc/[r]^{g^*u\circ
\alpha^*x} \ar@/_0.8pc/[r]_{g^*v\circ
\alpha^*x}\ar@{}[r]|{\Downarrow \,g^*\!\phi\circ 1_{\!\alpha^*\!x}}
&g^*y.}
$$
Thus, the construction $\D\mapsto \mathrm{hocolim}_C\D$ is the
function on objects of a 2-functor
$$
\mathrm{hocolim}_C-:\underline{\mathbf{2}\Cat}^{C^{\mathrm{op}}} \to
\underline{\mathbf{2}\Cat}^{\Delta^{\mathrm{op}}}.
$$

 The first two basic properties below quickly follow from the
 definition.

 \begin{proposition} Let $D:C\to
 \underline{\mathbf{2}\Cat}$ denote the constant $2$-functor on a $2$-category $C$ given by
 a $2$-category $D$. There is a natural homeomorphism
{\em $$\thb \mathrm{hocolim}_C\,D \thb \cong \tb D\tb \times \tb
C\tb.$$ }

In particular, for $D=\mathrm{pt}$ the terminal
$2$-category, {\em
$$\thb \mathrm{hocolim}_C\mathrm{ pt} \thb \cong
\tb C\tb,$$ } and for $C=\mathrm{ pt}$,
{\em
$$\thb \mathrm{hocolim}_\mathrm{ pt}\,D \thb \cong
\tb D\tb.$$ }
 \end{proposition}
\begin{proof} We have
 $$\thb \mathrm{hocolim}_C D\thb =| [p]\mapsto \tb D\times \ner_p C\tb|\cong
  | [p]\mapsto \tb D\tb \times |\ner_p C||\cong\tb D\tb\times |[p]\mapsto |\ner_p C||=\tb D\tb\times\tb C\tb.$$
\end{proof}

\begin{theorem}[Homotopy Invariance Theorem] \label{ith} Let $\D,\E:C\to \underline{\mathbf{2}\Cat}$ (or $\D,\E:C^{\mathrm{op}}\to \underline{\mathbf{2}\Cat}$) be
$2$-functors, where  $C$ is any $2$-category. If
$\Gamma:\D\Rightarrow \E$ is a $2$-transformation such that, for
each object $c$ of $C$, the $2$-functor $\Gamma_c:\D_c\to \E_c$ is a
weak equivalence of $2$-categories, then the induced map on
classifying spaces
$$
\Gamma_*:\thb \mathrm{hocolim}_C\D\thb \to \thb \mathrm{hocolim}_C\E\thb
$$
is a homotopy equivalence.
\end{theorem}
\begin{proof} Since $F$ is objectwise a weak equivalence, for any integer $p\geq 0$ the induced 2-functor
$$
\coprod_{c_0,\ldots,c_p}\hspace{-0.1cm} \D_{\!c_0}\!\times
C(c_0,c_1)\times \cdots\times C(c_{p-1},c_p) \longrightarrow
\coprod_{c_0,\ldots,c_p}\hspace{-0.1cm} \E_{c_0}\!\times
C(c_0,c_1)\times\cdots\times C(c_{p-1},c_p)
$$
is a weak equivalence. Then,

$ \thb \mathrm{hocolim}_C\D\thb=|[p]\mapsto \tb
\mathrm{hocolim}_C\D_p\tb| \simeq |[p]\mapsto \tb
\mathrm{hocolim}_C\E_p\tb|= \thb \mathrm{hocolim}_C\E\thb. $
\end{proof}

\begin{corollary}\label{cor1} Let $C$ be a $2$-category and  $\D:C\to
\underline{\mathbf{2}\Cat}$ (or $\D:C^{\mathrm{op}}\to
\underline{\mathbf{2}\Cat}$) a $2$-functor such that, for any object
$c$ of $C$, the $2$-category $\D_{\!c}$ is weakly contractible.
Then, the induced map by the collapse $2$-transformation
$\D\Rightarrow \mathrm{ pt}$,
$$\thb \mathrm{hocolim}_C\D\thb \to \thb
\mathrm{hocolim}_C\,\mathrm{pt}\thb\cong\tb C\tb,$$ is a homotopy
equivalence.
\end{corollary}

\section{The Grothendieck construction on 2-functors}

Since Thomason established his Homotopy Colimit Theorem, the
so-called ``Grothendieck construction" on diagrams of small
categories has become an essential tool in the homotopy theory of
classifying spaces. This construction underlies the 2-categorical
construction we treat here for 2-diagrams of 2-categories, which was
recently used by Cegarra in ~\cite[Theorem 4.5
$(i)$]{Cegarra-2011} to generalize Thomason's theorem for 2-diagrams
of categories, and also by Buckley in ~\cite[Theorem
2.2.11]{Buckley2014} to classify split (co)fibred 2-categories. For
a more general version of the enriched Grothendieck construction
below, which works even on lax bidiagrams of Benabou's bicategories,
we refer the reader to ~\cite{Bakovic2011}, ~\cite{Buckley2014},
~\cite{C-C-G2011} or ~\cite{CCH2014}.

Let $C$ be a 2-category, and let $\D:C \to
\underline{\mathbf{2}\Cat}$ be a 2-functor, whose  effect on cells
of $C$ is denoted as in \eqref{efcells}. The {\em Grothendieck
construction} on the 2-diagram $\D$ assembles the 2-diagram into a
2-category, denoted by
$$\xymatrix{\int_{C}\D},$$
whose objects are pairs $(a,x)$ with $a\in\Ob C$ and $x\in
\Ob\D_{\!a}$, the $1$-cells are pairs $(f,u):(a,x)\to (b,y)$, where
$f:a\to b$ is a 1-cell in $C$ and $u:f_*x\to y$ is a 1-cell in
$\D_b$, and the $2$-cells
$$\xymatrix@C=25pt{(a,x)\ar@{}[r]|{\Downarrow
(\alpha,\phi)}\ar@/^0.8pc/[r]^{(f,u)}\ar@/_0.8pc/[r]_{(g,v)}&(b,y),}
$$ are pairs  consisting of a 2-cell $\xymatrix@C=0.5pc{a
\ar@/^0.6pc/[rr]^{ f} \ar@/_0.6pc/[rr]_{
g}\ar@{}[rr]|{\Downarrow\alpha} &
 &b }$ of $C$ together with a 2-cell $\phi:u\Rightarrow v\circ \alpha_*x$ in
 $\D_b$,
 $$\xymatrix@C=0.8pc@R=-2pt{f_*x \ar@/^1pc/[rr]^{ u}
\ar@/_0.5pc/[rd]_(.4){ \alpha_*x}\ar@{}[rr]|{\Downarrow\phi} &
 &y.\\ &g_*x\ar@/_0.5pc/[ru]_{v}&}
$$

The vertical composition  of 2-cells
$$\xymatrix@C=16pt{(a,x) \ar[rr]_{
\Downarrow(\beta,\psi)}  \ar@/_1.4pc/[rr]_{
(h,w)} \ar@/^1.4pc/[rr]^{(f,u)}\ar@{}[rr]^{
\Downarrow(\alpha,\phi)} & &(b,y) }$$
is the 2-cell
$$ (\beta,\psi)\cdot (\alpha,\phi)=(\beta\cdot \alpha,
(\psi\circ 1_{\alpha_*x})\cdot \phi):(f,u)\Rightarrow
(h,w),$$ and the identity 2-cell of a 1-cell $(f,u)$ as above is $1_{(f,u)}=(1_f,1_u)$.

The horizontal composition  of two 1-cells
$\xymatrix@C=20pt{(a,x)\ar[r]^{(f,u)}&(b,y)\ar[r]^{(f',u')}&(c,z)}$
is the 1-cell
$$
(f',u')\circ (f,u)=(f'\circ f, u'\circ f'_*u):(a,x)\longrightarrow (c,z),
$$
the identity 1-cell of an object $(a,x)$ is $1_{(a,x)}=(1_a,1_x)$,
and the horizontal composition of 2-cells
$$
\xymatrix@C=30pt{(a,x)\ar@{}[r]|{\Downarrow (\alpha,\phi)}\ar@/^1pc/[r]^{(f,u)}
\ar@/_1pc/[r]_{(g,v)}&(b,y)
\ar@{}[r]|{\Downarrow (\alpha',\phi')}\ar@/^1pc/[r]^{(f',u')}\ar@/_1pc/[r]_{(g',v')}&(c,z)}
$$
is the 2-cell
$$
(\alpha',\phi')\circ (\alpha,\phi)=(\alpha'\circ \alpha, \phi'\circ f'_*\phi):(f'\circ f, u'\circ f'_*u)
\Rightarrow (g'\circ g, v'\circ g'_*v).
$$

Note that a 2-transformation $\Gamma:\D\Rightarrow\E$ between
2-functors $\D,\E:C\to \underline{\mathbf{2}\Cat}$ induces the
2-functor $\int_C\Gamma: \int_C\D\to \int_C\E$ such that
$$\begin{array}{ccc}
\xymatrix@C=3pc{
(a,x)\ar@/^1pc/[r]^{(f,u)} \ar@/_1pc/[r]_{(g,v)}
  \ar@{}[r]|{\Downarrow(\alpha,\phi)} & (b,y)}
& \mapsto &\xymatrix@C=3pc{
(a,\Gamma_a   x)
  \ar@/^1pc/[r]^{(f,\Gamma_bu)}
  \ar@/_1pc/[r]_{(g,\Gamma_bv)}
  \ar@{}[r]|{\Downarrow (\alpha, \Gamma_b\phi)}
& (b, \Gamma_by). }
\end{array}
$$
Also, for $\Gamma':\D\Rightarrow \E$ any other  2-transformation, a
modification $m:\Gamma \Rrightarrow \Gamma'$ gives rise to the
2-transformation $ \int_C m:\int_C \Gamma \Rightarrow \int_C \Gamma'
$ given by
 $$\xymatrix{ \int_C m(a,x)=(1_a, m_a x):(a,\Gamma_ax)\to  (a,\Gamma'_a x). }$$

Thus, the 2-categorical Grothendieck construction provides a
2-functor
$$\xymatrix@C=20pt{\int_C-\,:\, \underline{\mathbf{2}\Cat}^{C}\ar[r]&
\underline{\mathbf{2}\Cat}.}$$

In a similar way, if $\D:C^{\mathrm{op}} \to \underline{\mathbf{2}\Cat}$ is a
2-functor, the  {\em Grothendieck construction on $\D$} is the
2-category, denoted by $\xymatrix{\int_{C}\D}$,  whose objects are
pairs $(a,x)$ with $a\in\Ob C$ and $x\in \Ob\D_{\!a}$, whose
$1$-cells  $(f,u):(a,x)\to (b,y)$ are pairs where $f:a\to b$ is a
1-cell in $C$ and $u:x\to f^*y$ is a 1-cell in $\D_a$, and whose
$2$-cells $(\alpha,\phi):(u,f)\Rightarrow(v,g)$ are pairs consisting
of a 2-cell $\alpha:u\Rightarrow v$ of $C$
 together with a 2-cell $\phi:\alpha^*y\circ u\Rightarrow v$ in
 $\D_b$,
 $$\xymatrix@C=1pc@R=-2pt{&f^*y\ar@/^0.5pc/[rd]^{\alpha^*y}& \\
 x \ar@/^0.5pc/[ru]^{ u}
\ar@/_1pc/[rr]_(.5){v}\ar@{}[rr]|{\Downarrow\phi} &
 &g^*y.}
$$

The vertical composition  of the 2-cell $(\alpha,\phi)$ as above
with a 2-cell $(\beta,\psi):(v,g)\Rightarrow  (w,h)$ is the 2-cell
$$ (\beta,\psi)\cdot (\alpha,\phi)=(\beta\cdot \alpha,
\psi\cdot (1_{\beta^*y}\circ \phi)):(f,u)\Rightarrow
(h,w).$$
The horizontal composition  of two 1-cells
$\xymatrix@C=20pt{(a,x)\ar[r]^{(f,u)}&(b,y)\ar[r]^{(f',u')}&(c,z)}$
is the 1-cell
$$
(f',u')\circ (f,u)=(f'\circ f,  f^*u'\circ u):(a,x)\longrightarrow (c,z),
$$
and the horizontal composition of 2-cells
$$
\xymatrix@C=30pt{(a,x)\ar@{}[r]|{\Downarrow (\alpha,\phi)}\ar@/^1pc/[r]^{(f,u)}
\ar@/_1pc/[r]_{(g,v)}&(b,y)
\ar@{}[r]|{\Downarrow (\alpha',\phi')}\ar@/^1pc/[r]^{(f',u')}\ar@/_1pc/[r]_{(g',v')}&(c,z)}
$$
is the 2-cell
$$
(\alpha',\phi')\circ (\alpha,\phi)=(\alpha'\circ \alpha, f^*\phi'\circ \phi):
(f'\circ f, f^*u'\circ u)
\Rightarrow (g'\circ g, g^*v'\circ v).
$$

Thus, similarly to the covariant case,  $\D\mapsto \int_C\D$ is the function on objects of a 2-functor
$$\xymatrix@C=20pt{\int_C-\,:\, \underline{\mathbf{2}\Cat}^{C^{\mathrm{op}}}\ar[r]&
\underline{\mathbf{2}\Cat}.}$$

\section{The Homotopy Colimit Theorem for 2-functors}

This section is fully dedicated to proving Theorem \ref{hct} below,
which has Thomason's Homotopy Colimit Theorem ~\cite[Theorem
1.2]{Thomason1980} as its natural precedent and also includes the
results in ~\cite[Theorem 4.5]{Cegarra-2011} as particular cases.

In the proof we give of this 2-categorical Homotopy Colimit Theorem,
we use the $\overline{W}$-construction on a bisimplicial set
by Artin and Mazur ~\cite[\S III]{Artin-Mazur1966}, also called its
``codiagonal" or ``total complex".  Recall that, by viewing a
bisimplicial set $S:\Delta^{\!{\mathrm{op}}}\times
\Delta^{\!{\mathrm{op}}}\to \Set$ as a horizontal simplicial object
in the category of vertical simplicial sets, then the set of
$n$-simplices of $\overline{W}S$ is
$$
\Big\{(t_{n,0}, \dots,t_{0,n})\in
\prod_{p+q=n}\hspace{-4pt}S_{p,q}~~|~d^h_0t_{p,q}=d^v_{q+1}t_{p-1,q+1}
\text{ for } \, n\geq p\geq 1 \Big\}$$ and, for $0\leq i\leq n$, the
faces and degeneracies of an $n$-simplex are given by
$$
\begin{array}{lll}d_i(t_{n,0},
\dots,t_{0,n})&=&(d_i^ht_{n,0},\dots,d_1^ht_{n-i+1,i-1},d_i^vt_{n-i-1,i+1},\dots,d_i^vt_{0,n}),
\\[-4pt]~\\
s_i(t_{n,0}, \dots,t_{0,n})&=&(s_i^ht_{n,0},\dots,s_0^ht_{n-i,i},s_i^vt_{n-i,i},\dots,s_i^vt_{0,n}).
\end{array}
$$
There is a natural Alexander-Whitney-type diagonal approximation
$\Diag\,S\to \overline{W}S$,
$$
S_{n,n}\ni t\mapsto \big((d_1^v)^nt,(d_2^v)^{n-1}d_0^ht,\dots, (d_{p+1}^v)^{n-p}(d_0^h)^pt,\dots,
(d_0^h)^nt\big),
$$
inducing a homotopy equivalence on geometric realizations (see
~\cite{C-R2005}, ~\cite{Stevenson2012}, or ~\cite{Zisman2014} for a
proof).

\begin{equation}\label{diwb}|\Diag\,S|\simeq |\overline{W}S|.\end{equation}

\begin{theorem}\label{hct}  For any $2$-functor
$\D: C\to \underline{\mathbf{2}\Cat}$ (or $\D: C^{\mathrm{op}}\to
\underline{\mathbf{2}\Cat}$), where $C$ is a $2$-category, there
exists a natural homotopy equivalence
$$\xymatrix{\thb\mbox{\em hocolim}_C\D\thb \simeq \tb
\int_C\D\tb.}$$
\end{theorem}
\begin{proof}
We shall treat the covariant case, as the other is proven
similarly.

For  any 2-category $C$, we have a natural homotopy equivalence
$$\tb C\tb=|\Diag \ner\ner C|\simeq |\W \ner\ner C|.$$
To describe this simplicial set $\W\ner\ner C$, let us first
represent a $(p,q)$-simplex of the bisimplicial set $\ner\ner C$ as
a diagram $(c,f,\alpha)_{p,q}$ in $C$ of the form
\begin{equation}\label{pqnn}
\xymatrix@C=10pt {(c,f,\alpha)_{p,q}:&c_q
\ar@/^1pc/[rrr]
\ar@{}[rrr]|{\vdots}
 \ar@/_1pc/[rrr]
\ar@/^2.5pc/[rrr]^{f_{q+1}^0}_(.5)*+{\Downarrow\scriptstyle{\alpha_{q+1}^1}}
\ar@/_2.5pc/[rrr]_-{f^q_{q+1}}^(.5)*+{\Downarrow\scriptstyle{\alpha^q_{q+1}}}
&&& c_{q+1}\ar@/^1pc/[rrr]
\ar@{}[rrr]|{\vdots}
 \ar@/_1pc/[rrr]
\ar@/^2.5pc/[rrr]^{f_{q+2}^0}_(.5)*+{\Downarrow\scriptstyle{\alpha_{q+2}^1}}
\ar@/_2.5pc/[rrr]_-{f^q_{q+2}}^(.5)*+{\Downarrow\scriptstyle{\alpha^q_{q+2}}}& &&c_{q+2}& \cdots&
 c_{q+p-1} \ar@/^1pc/[rrr]
\ar@{}[rrr]|{\vdots}
 \ar@/_1pc/[rrr]
\ar@/^2.5pc/[rrr]^{f_{q+p}^0}_(.5)*+{\Downarrow\scriptstyle{\alpha_{q+p}^1}}
\ar@/_2.5pc/[rrr]_-{f_{q+p}^{q}}^(.5)*+{\Downarrow\scriptstyle{\alpha_{q+p}^{q}}}&&& c_{q+p},}
\end{equation}
whose horizontal $i$-face is obtained by deleting the object
$c_{q+i}$ and using, for $0<i<p$, the composite cells
$f^k_{q+i+1}\circ f^k_{q+i}$ and $\alpha^k_{q+i+1}\circ
\alpha^k_{q+i}$ to rebuild the new $(p-1,q)$-simplex, and whose
vertical $j$-face is obtained by deleting all the 1-cells
$f^j_{q+m}$ and using the composite 2-cells $\alpha^{j+1}_{q+m}\cdot
\alpha^{j}_{q+m}$, for $0<j<q$, to complete the $(p,q-1)$ simplex of
$\ner\ner B$. Then,  it is straightforward to obtain the following
description of the simplicial set $\W\ner\ner C$:
 The vertices are the objects $c_0$ of $C$, the
1-simplices are the 1-cells $f_1^0:c_0\to c_1$ of $C$, and, for
$n\geq 2$, the $n$-simplices are diagrams $(c,f,\alpha)_n$ in $C$ of
the form
\begin{equation}\label{simbarw}
\xymatrix @C=10pt {(c,f,\alpha)_n:&c_0\ar[rr]^{f_1^0} && c_1 \ar@/^1pc/[rr]^{f_2^0} \ar@/_1pc/[rr]_-{f_2^1}
& {\Downarrow \scriptstyle{\alpha_2^1}} &
c_2 \ar[rrr]|{f_3^1}
\ar@/^2pc/[rrr]^{f_3^0}_(.4)*+{\Downarrow\scriptstyle{\alpha_3^1}}
\ar@/_2pc/[rrr]_-{f_3^2}^(.4)*+{\Downarrow \scriptstyle{\alpha_3^2}}& && c_3 &\cdots&c_{n-1}
\ar@/^1pc/[rrr]
\ar@{}[rrr]|{\vdots}
 \ar@/_1pc/[rrr]
\ar@/^2.5pc/[rrr]^{f_n^0}_(.5)*+{\Downarrow\scriptstyle{\alpha_n^1}}
\ar@/_2.5pc/[rrr]_-{f_n^{n-1}}^(.5)*+{\Downarrow\scriptstyle{\alpha_n^{n-1}}}&&& c_n,
}
\end{equation}
that is, they consist of objects $c_m$ of $C$, $0\leq m\leq n$,
1-cells $f_m^k:c_{m-1}\to c_{m}$, $0\leq k <m\leq n$, and 2-cells
$\alpha_m^k:f_m^{k-1}\Rightarrow f_m^{k}$, $0<k<m\leq n$. The
simplicial operators of $\W\ner\ner C$ act much as for the usual nerve
of an ordinary category: The $i$-face of an $n$-simplex as in
\eqref{simbarw} is obtained by deleting the object $c_i$ and the
$1$-cells $f_m^{i}:c_{m-1}\to c_m$, for $i< m$, and using the
composite 1-cells $f^k_{i+1}\circ f^k_i:c_{i-1}\to c_{i+1}$, $ k<i$,
the horizontally composite 2-cells $\alpha^k_{i+1}\circ
\alpha^k_i:f^{k-1}_{i+1}\circ f^{k-1}_i\Rightarrow f^k_{i+1}\circ
f^k_i$, $0<k<i$, and the vertically composed 2-cells
$\alpha^{i+1}_m\cdot \alpha^i_m:f^{i-1}_m\Rightarrow f^{i+1}_m$, when
$i<m-1$, to complete the new $(n-1)$-simplex. The $i$-degeneracy
of $(c,f,\alpha)_n$ is constructed by repeating the object $c_i$ at
the $i + 1$-place and inserting $i+1$ times the identity 1-cell
$1_{c_i}:c_i\to c_i$, $i$ times the identity 2-cell
$1_{1_{c_i}}:1_{c_i}\Rightarrow 1_{c_i}$ and, for each $i<m$, by
replacing the 1-cell $f_m^i:c_{m-1}\to c_m$ by the identity 2-cell
$1_{f_m^i}:f_m^i\Rightarrow f_m^i$.

Now, for $\D:C\to \underline{\mathbf{2}\Cat}$ any given 2-functor,
we have
\begin{align}\label{111} \thb\mbox{hocolim}_B\D\thb
&=|\Diag \ner\ner\mbox{hocolim}_C\D|= |\Diag \big([p]\mapsto \Diag
\ner\ner\mbox{hocolim}_C\D_p\big)|\\ &\simeq |\Diag \big([p]\mapsto \W\ner\ner
\mbox{hocolim}_C\D_p\big)|\simeq |\W([p]\mapsto
\W\ner\ner\text{hocolim}_C\D_p)|,\nonumber
\end{align}
and the last simplicial set $\W([p]\mapsto
\W\ner\ner\text{hocolim}_C\D_p)$ can be described as follows: Its
$n$-simplices are pairs
\begin{equation}\label{twotri}((c,f,\alpha)_n,(x,u,\phi)_n)\end{equation} where $(c,f,\alpha)_n$ is an
$n$-simplex of $\W \ner\ner C$ as in \eqref{simbarw}, whereas
$(x,u,\phi)_n$ is a list with a diagram in each 2-category
$D_{\!c_0}$,...,$D_{\!c_n}$ of the form
\begin{equation}\label{simbarww}
x_0,\xymatrix@C=14pt{f_{1*}^0x_0\ar[r]^{u_1^0}&x_1}, \xymatrix@C=1pt{f_{2*}^1x_1 \ar@/^1pc/[rr]^{u_2^0}
 \ar@/_1pc/[rr]_-{u_2^1}
& {\Downarrow \scriptstyle{\phi_2^1}} &
x_2}, \xymatrix@C=4pt{f_{3*}^2x_2 \ar[rrr]|{u_3^1}
\ar@/^1.8pc/[rrr]^{u_3^0}_(.4)*+{\Downarrow\scriptstyle{\phi_3^1}}
\ar@/_1.8pc/[rrr]_-{u_3^2}^(.4)*+{\Downarrow \scriptstyle{\phi_3^2}}& && x_3}, \dots,
\xymatrix@C=4pt{f^{n-1}_{n*}x_{n-1}
\ar@/^1pc/[rrr]
\ar@{}[rrr]|{\vdots}
 \ar@/_1pc/[rrr]
\ar@/^2.5pc/[rrr]^{u_n^0}_(.5)*+{\Downarrow\scriptstyle{\phi_n^1}}
\ar@/_2.5pc/[rrr]_-{u_n^{n-1}}^(.5)*+{\Downarrow\scriptstyle{\phi_n^{n-1}}}&&&x_n.
}
\end{equation}
That is,  $(x,u,\phi)_n$ consists of 0-cells $x_k$ of $D_{\!c_k}$,
$0\leq k\leq n$, 1-cells $u_m^k:f^{m-1}_{m*}x_{m-1}\to x_{m}$,
$0\leq k <m\leq n$, and 2-cells $\phi_m^k:u_m^{k-1}\Rightarrow
u_m^{k}$, $0<k<m\leq n$. Further, the $i$-face of the simplex
$((c,f,\alpha)_n,(x,u,\phi)_n)$ is obtained by taking the $i$-face
of $(c,f,\alpha)_n$ in the simplicial set $\W\ner\ner C$ and, in a
similar way, by deleting the object $x_i$ and the $1$-cells
$u_m^{i}:f_{m*}^{m-1}x_{m-1}\to x_m$, for $i< m$, and then using the
composite 1-cells
$$
\xymatrix@C=40pt{f_{i+1*}^{i-1}f^{i-1}_{i*}x_{i-1}\ar[r]^-{{\alpha^i_{i+1*}}f^{i-1}_{i*}x_{i-1}}&
f_{i+1*}^{i}f^{i-1}_{i*}x_{i-1}\ar[r]^-{f_{i+1*}^{i} u^k_i}&f_{i+1*}^{i} x_i
\ar[r]^-{u^k_{i+1}}&x_{i+1}}
$$
for $k<i$, the horizontally composite 2-cells
$$
\xymatrix@C=40pt{f_{i+1*}^{i-1}f^{i-1}_{i*}x_{i-1}\ar@/^1.5pc/[r]^-{\alpha^i_{i+1*}f^{i-1}_{i*}x_{i-1}}
\ar@/_1.5pc/[r]_-{\alpha^i_{i+1*}f^{i-1}_{i*}x_{i-1}}\ar@{}[r]|(.55){\Downarrow 1}&
f_{i+1*}^{i}f^{i-1}_{i*}x_{i-1}\ar@/^1.5pc/[r]^-{f_{i+1*}^{i} u^{k-1}_i}
\ar@/_1.5pc/[r]_-{f_{i+1*}^{i} u^k_i}\ar@{}[r]|(.55){\Downarrow f^i_{i+1*}\phi^k_i}&f_{i+1*}^{i} x_i
\ar@/^1.5pc/[r]^-{u^{k-1}_{i+1}}\ar@/_1.5pc/[r]_-{u^k_{i+1}}\ar@{}[r]|(.55){\Downarrow \phi^k_{i+1}}&x_{i+1}}
$$
for $0<k<i$, and the vertically composed 2-cells $\phi^{i+1}_m\cdot
\phi^i_m:u^{i-1}_m\Rightarrow u^{i+1}_m$, $i<m-1$, to complete the
new $(n-1)$-simplex. Similarly, the $i$-degeneracy of
$((c,f,\alpha)_n,(x,u,\phi)_n)$ is given by first taking the
$i$-degeneracy of of $(c,f,\alpha)_n$ in the simplicial set
$\W\ner\ner C$ and second by repeating the object $x_i$, inserting
$i+1$ times the identity 1-cell $1_{x_i}:x_i\to x_i$, $i$ times the
identity 2-cell $1_{1_{x_i}}:1_{x_i}\Rightarrow 1_{x_i}$ and, for
each $i<m$, by replacing the 1-cell $u_m^{i}$ by the identity 2-cell
$1_{u_m^i}:u_m^i\Rightarrow u_m^i$.

There is another way to get the same simplicial set $\W([p]\mapsto
\W\ner\ner\text{hocolim}_C\D_p)$, which is as follows:

Let $E$ be the trisimplicial set whose $(p,n,q)$-simplices are pairs
\begin{equation}\label{bipqn}((c,f,\alpha)_{p,q},(x,u,\phi)_{p,n,q})\end{equation} with
$(c,f,\alpha)_{p,q}$ a $(p,q)$-simplex of $\ner\ner C$, as in
\eqref{pqnn}, and $(x,u,\phi)_{p,n,q}$ a system of data consisting
of a diagram in each 2-category $\D_{\!c_{q+1}}$, ..., $\D_{\!c_{q+p}}$
of the form
$$
\xymatrix@C=10pt{(x,u,\phi)_{p,n,q}:&f^q_{q+1*}x_q
\ar@/^1pc/[rrr]
\ar@{}[rrr]|{\vdots}
 \ar@/_1pc/[rrr]
\ar@/^2.5pc/[rrr]^{u_{q+1}^0}_(.5)*+{\Downarrow\scriptstyle{\phi_{q+1}^1}}
\ar@/_2.5pc/[rrr]_-{u^n_{q+1}}^(.5)*+{\Downarrow\scriptstyle{\phi^n_{q+1}}}
&&& x_{q+1},} \ldots,
\xymatrix@C=10pt {
 f^q_{q+p\,*}x_{q+p-1} \ar@/^1pc/[rrr]
\ar@{}[rrr]|{\vdots}
 \ar@/_1pc/[rrr]
\ar@/^2.5pc/[rrr]^{u_{q+p}^0}_(.5)*+{\Downarrow\scriptstyle{\phi_{q+p}^1}}
\ar@/_2.5pc/[rrr]_-{u_{q+p}^{n}}^(.5)*+{\Downarrow\scriptstyle{\phi_{q+p}^{n}}}&&& x_{q+p}.}
$$
That is, it consists of objects $x_{q}\in D_{\!c_q}$, $\ldots$,
$x_{q+p}\in \D_{\!c_{q+p}}$, 1-cells
$u^k_{q+m}:f^q_{q+m*}x_{q+m-1}\to x_{q+m}$, $0\leq k\leq n$, $0<
m\leq p$, and 2-cells $\phi^k_{q+m}:u^{k-1}_{q+m}\Rightarrow
u^{k}_{q+m}$, $0< k\leq n$, $0< m\leq p$.

The $i$-face in the $p$-direction map of $E$ carries
the $(p,n,q)$-simplex \eqref{bipqq} to the $(p-1,n,q)$-simplex  obtained by taking the horizontal
$i$-face of $(c,f,\alpha)_{p,q}$ in $\ner\ner C$, deleting the
object $x_{q+i}$, and  using the composite 1-cells
$$
\xymatrix@C=40pt{f^{q}_{q+i+1*}f^{q}_{q+i*}x_{q+i-1}\ar[r]^-{f_{q+i+1*}^{q} u^k_{q+i}}&f_{q+i+1*}^{q} x_{q+i}
\ar[r]^-{u^k_{q+i+1}}&x_{q+i+1}}
$$
and the horizontally composite 2-cells
$$
\xymatrix@C=60pt{
f_{q+i+1*}^{q}f^{q}_{q+i*}x_{q+i-1}\ar@/^1.5pc/[r]^-{f_{q+i+1*}^{q} u^{k-1}_{q+i}}
\ar@/_1.5pc/[r]_-{f_{q+i+1*}^{q} u^k_{q+i}}\ar@{}[r]|(.56){\Downarrow f^q_{q+i+1*}\phi^k_{q+i}}&f_{q+i+1*}^{q}
x_{q+i}
\ar@/^1.5pc/[r]^-{u^{k-1}_{q+i+1}}\ar@/_1.5pc/[r]_-{u^k_{q+i+1}}
\ar@{}[r]|(.55){\Downarrow \phi^k_{q+i+1}}&x_{q+i+1}}
$$ to complete the new $(p-1,n,q)$-simplex.

The $j$-face in the $n$-direction of the $(p,n,q)$-simplex \eqref{bipqq} is obtained  by keeping $(c,f,\alpha)_{p,q}$
unaltered, deleting all the $1$-cells $u^j_{q+m}$, and using, when
$0<j<n$, the composite 2-cells $\phi^{j+1}_{q+m}\cdot
\phi^{j}_{q+m}$ to complete the simplex.

 For any $k<q$, the $k$-face in the $q$-direction of the $(p,n,q)$-simplex \eqref{bipqq}
is given by replacing $(c,f,\alpha)_{p,q}$ by its
vertical $k$-face in $\ner\ner C$ and keeping
$(x,u,\phi)_{p,n,q}$ unchanged, while its $q$-face
consists of  the vertical $q$-face of $(c,f,\alpha)_{p,q}$  in $\ner\ner
C$ (which, recall, is obtained by deleting the $1$-cells
$f^q_{q+m}$) together with the list of diagrams
$$
\xymatrix@C=12pt{f^{q-1}_{q+1*}x_q
\ar@/^1pc/[rrr]
\ar@{}[rrr]|{\vdots}
 \ar@/_1pc/[rrr]
\ar@/^2.5pc/[rrr]^{u_{q+1}^0\circ \alpha^q_{q+1*}x_q}_(.5)*+{\Downarrow\scriptstyle{\phi_{q+1}^1\circ 1}}
\ar@/_2.5pc/[rrr]_-{u^n_{q+1}\circ \alpha^q_{q+1*}x_q}^(.5)*+{\Downarrow\scriptstyle{\phi^n_{q+1}\circ 1}}
&&& x_{q+1},} \ldots,
\xymatrix@C=10pt {
 f^{q-1}_{q+p\,*}x_{q+p-1} \ar@/^1pc/[rrr]
\ar@{}[rrr]|{\vdots}
 \ar@/_1pc/[rrr]
\ar@/^2.5pc/[rrr]^{u_{q+p}^0\circ \alpha^q_{q+1*}x_q}_(.5)*+{\Downarrow\scriptstyle{\phi_{q+p}^1}\circ 1}
\ar@/_2.5pc/[rrr]_-{u_{q+p}^{n}\circ \alpha^q_{q+1*}x_q}^(.5)*+{\Downarrow\scriptstyle{\phi_{q+p}^{n}
\circ 1}}
&&& x_{q+p}.}
$$

With degeneracies given in a standard way, it is straightforward to see that $E$ is a trisimplicial set. Then, an easy verification shows that
\begin{equation}\label{112}
\W([p]\mapsto \Diag E_{p,\bullet,\bullet})\cong\W([p]\mapsto
\W\ner\ner\text{hocolim}_C\D_p),
\end{equation}
and therefore we have homotopy equivalences
\begin{equation}\label{113} \thb\mbox{hocolim}_C\D\thb
\overset{\eqref{111}\eqref{112}}\simeq |\W([p]\mapsto \Diag E_{p,\bullet,\bullet})|\overset{\eqref{diwb}}\simeq |\W([p]\mapsto \W
E_{p,\bullet,\bullet})|.\end{equation}

Now, an analysis of the simplicial set $\W E_{p,\bullet,\bullet}$
says that its $q$-simplices are pairs
\begin{equation}\label{bipqq}((c,f,\alpha)_{p,q},(x,u,\phi)_{p,q})\end{equation}
with a $(p,q)$-simplex $(c,f,\alpha)_{p,q}$ of $\ner\ner C$, as in
\eqref{pqnn}, together with  data $(x,u,\phi)_{p,q}$ consisting of a
diagram in each 2-category $\D_{\!c_{q+1}}$, ..., $\D_{\!c_{q+p}}$ of
the form
$$
\xymatrix@C=50pt@R=16pt{f^0_{q+1*}x_q\ar@<-2pt>[d]_{\alpha^1_{q+1*}x_q}\ar@<3pt>[rdd]^{u^0_{q+1}}& \\
f^1_{q+1*}x_q\ar[rd]
\ar@{}[r]|(.4){\Downarrow\phi^1_{q+1}}
\ar@{}[dd]|{\vdots}&\\
&x_{q+1}, \\
f^{q-1}_{q+1*}x_q\ar@<-2pt>[d]_{\alpha^q_{q+1*}x_q}\ar[ru]
\ar@{}[r]|(.4){\Downarrow\phi^q_{q+1}}
& \\
f^q_{q+1*}x_q\ar@<-3pt>[ruu]_{u^q_{q+1}}& }
\xymatrix@C=3pt@R=15pt{ &\\ & \\ &\\&\ldots, \\ &}
\xymatrix@C=50pt@R=16pt{f^0_{q+p*}x_{q+p-1}\ar@<-2pt>[d]_{\alpha^1_{q+p*}x_{q+p-1}}\ar@<3pt>[rdd]^{u^0_{q+p}}& \\
f^1_{q+p*}x_{q+p-1}\ar[rd]
\ar@{}[r]|(.4){\Downarrow\phi^1_{q+p}}
\ar@{}[dd]|{\vdots}& \\
&x_{q+p}. \\
f^{q-1}_{q+p*}x_{q+p-1}\ar@<-2pt>[d]_{\alpha^q_{q+p*}x_{q+p-1}}\ar[ru]
\ar@{}[r]|(.4){\Downarrow\phi^q_{q+p}}
& \\
f^q_{q+p*}x_{q+p-1}\ar@<-3pt>[ruu]_{u^q_{q+p}}&}
$$
More precisely, $(x,u,\phi)_{p,q}$  consists of objects $$x_{q}\in
\D_{\!c_q}, \ldots, x_{q+p}\in \D_{\!c_{q+p}},$$ 1-cells
$$u^k_{q+m}:f^k_{q+m*}x_{q+m-1}\to x_{q+m}, \hspace{0.4cm}(0\leq k\leq q,\, 0<
m\leq p)$$ and 2-cells $$\phi^k_{q+m}:u^{k-1}_{q+m}\Rightarrow
 u^{k}_{q+m}\circ \alpha^k_{q+m*}x_{q+m-1} \hspace{0.4cm} (0<
k\leq n,\, 0< m\leq p).$$ The $j$-face of such a $q$-simplex
\eqref{bipqq} is given by taking the vertical $j$-face of
$(c,f,\alpha)_{p,q}$ in $\ner\ner C$, deleting the 1-cells
$u^j_{q+m}$, and inserting the pasted 2-cells below, for $0<j<q$.
$$
\xymatrix{f^{j-1}_{q+m*}x_{q+m-1}
\ar[dd]_{(\alpha^{j+1}_{q+m}\cdot \alpha^j_{q+m})_*x_{q+m-1}}
\ar@/^1.5pc/[rrd]^{u^{j-1}_{q+m}}\ar[rd]|{\alpha^j_{q+m*}x_{q+m-1}}&
\ar@{}[d]|(.5){\Downarrow \phi^j_{q+m}}& \\
\ar@{}[r]|(.35){=} &f^{j}_{q+m*}x_{q+m-1}\ar[r]\ar@{}[d]|(.5){\Downarrow \phi^{j+1}_{q+m}}
\ar[ld]|{\alpha^{j+1}_{q+m*}x_{qm-1}}&x_{q+m} \\
f^{j+1}_{q+m*}x_{q+m-1}\ar@/_1.5pc/[rru]_{u^{j+1}_{q+m}}&&
}
$$

Then, an easy and straightforward verification shows that an
$n$-simplex of the simplicial set $\W([p]\mapsto
\W E_{p,\bullet,\bullet})$ is a pair
\begin{equation}\label{twotri2}((c,f,\alpha)_n,(x,u,\phi)_n),\end{equation}
where $(c,f,\alpha)_n$ is an $n$-simplex of $\W \ner\ner C$ as in
\eqref{simbarw}, while $(x,u,\phi)_n$ is a list with a diagram in
each 2-category $\D_{\!c_0}$,...,$\D_{\!c_n}$ of the form
$$
\xymatrix@R=3pt@C=10pt{&&&&&&&&&f^0_{n*}x_{n-1}\ar@{}@<18pt>[dd]|(.65){\Downarrow\phi^1_n}\ar[dd]\ar@{}@<3pt>[dd]_{\alpha^1_{n*}x_{n-1}}
\ar[rrrddd]^{u^0_n}&&\\
& & & & &f^0_{3*}x_2\ar[dd]\ar@{}@<3pt>[dd]_{\alpha^1_{3*}x_2}\ar[rrdd]^{u^0_3}
\ar@{}@<14pt>[dd]|(.6){\Downarrow\phi^1_3}&&&&&&\\
&&&f^0_{2*}x_1\ar[dd]\ar@{}@<3pt>[dd]_{\alpha^1_{2*}x_1}\ar[rd]^{u^0_2}\ar@{}@<12pt>[dd]|{\Downarrow\phi^1_2}
&&&&&&f^1_{n*}x_{n-1}\ar[rrrd]\ar@{}[dd]|{\vdots}&&\\
x_0,&\hspace{-10pt}f^0_{1*}x_0\ar[r]^{u^0_1}&x_1,&\hspace{-10pt} &x_2,&f^1_{3*}x_2\ar[rr]
\ar@{}@<-2pt>[rr]^{u^1_3}
\ar@{}@<14pt>[dd]|(.35){\Downarrow\phi^2_3}
\ar[dd]\ar@{}@<3pt>[dd]_{\alpha^2_{3*}x_2}&&x_3,&\ldots,&\hspace{-10pt}&&&x_n,\\
&&&f^1_{2*}x_1\ar[ru]_{u^1_2}&&&&&&f^{n-2}_{n*}x_{n-1}
\ar@{}@<18pt>[dd]|(.35){\Downarrow\phi^{n-1}_n}
\ar[rrru]\ar[dd]\ar@{}@<3pt>[dd]_{\alpha^{n-1}_{n*}x_{n-1}}
&&&\\
&&&&&f^2_{3*}x_2\ar[rruu]_{u^2_3}&&&&&&\\
&&&&&&&&&f^{n-1}_{n*}x_{n-1}\ar[rrruuu]_{u^{n-1}_n}&&
}
$$
That is,  $(x,u,\phi)_n$ consists of 0-cells
$$x_{0}\in
\D_{\!c_0}, \ldots, x_{n}\in \D_{\!c_{n}},$$ 1-cells
$$u^k_{m}:f^k_{m*}x_{m-1}\to x_{m}, \hspace{0.4cm}(0\leq k<m\leq n)$$
and 2-cells
$$\phi^k_{m}:u^{k-1}_{m}\Rightarrow
 u^{k}_{m}\circ \alpha^k_{m*}x_{m-1} \hspace{0.4cm} (0<k<m\leq n).$$

 Further, the $i$-face of the simplex
$((c,f,\alpha)_n,(x,u,\phi)_n)$ is obtained by taking the $i$-face
of $(c,f,\alpha)_n$ in the simplicial set $\W\ner\ner C$ and, in a
similar way, by deleting the object $x_i$ and all the $1$-cells
$u_m^{i}:f_{m*}^{i}x_{m-1}\to x_m$, for $i< m$, and then using the
composite 1-cells
$$
\xymatrix@C=40pt{f^{k}_{i+1*}f^{k}_{i*}x_{i-1}\ar[r]^-{f_{i+1*}^{k} u^k_{i}}&f_{i+1*}^{k} x_{i}
\ar[r]^-{u^k_{i+1}}&x_{i+1}}
$$
for $k<i$, and the pasted 2-cells
$$
\xymatrix{f^{k-1}_{i+1*}f^{k-1}_{i*}x_{i-1}\ar[dd]_{(\alpha^k_{i+1}\circ \alpha^k_i)_*x_{i-1}}
\ar[dr]|{f^{k-1}_{i+1*}\alpha^{k}_{i*}x_{i-1}}
\ar[rr]^-{f^{k-1}_{i+1*}u^{k-1}_i}&  \ar@{}[d]|(.35){\Downarrow f^{k-1}_{i+1*}\phi^k_i}&
f^{k-1}_{i+1*}x_{i}\ar[rd]^{u^{k-1}_{i+1}}\ar[dd]_{\alpha^{k}_{i+1*}x_i}
 \ar@{}@<24pt>[dd]|(.5){\Downarrow \phi^k_{i+1}}& \\
\ar@{}[r]|(.35){=} & f^{k-1}_{i+1*}f^{k}_{i*}x_{i-1}
\ar@{}[rd]|(.4){=}
 \ar[dl]|{\alpha^{k}_{i+1*}f^{k}_{i*}x_{i-1}}
 \ar[ru]|{f^{k-1}_{i+1*}u^k_i}& & x_{i+1} \\
 f^{k}_{i+1*}f^{k}_{i*}x_{i-1}\ar[rr]^-{f^{k}_{i+1*}u^{k}_i}&&f^{k}_{i+1*}x_{i}
\ar[ru]_{u^k_{i+1}}
 &
}
$$
for $0<k<i$, and
$$
\xymatrix{f^{i-1}_{m*}x_{m-1}
\ar[dd]_{(\alpha^{i+1}_{m}\cdot \alpha^i_m)_*x_{m-1}}
\ar@/^1.5pc/[rrd]^{u^{i-1}_m}\ar[rd]|{\alpha^i_{m*}x_{m-1}}&
\ar@{}[d]|(.5){\Downarrow \phi^i_{m}}& \\
\ar@{}[r]|(.35){=} &f^{i}_{m*}x_{m-1}\ar[r]\ar@{}[d]|(.5){\Downarrow \phi^{i+1}_{m}}\ar[ld]|{\alpha^{i+1}_{m*}x_{m-1}}&x_m \\
f^{i+1}_{m*}x_{m-1}\ar@/_1.5pc/[rru]_{u^{i+1}_m}&&
}
$$
for $0<i<m-1$, to complete the $i$-face $(n-1)$ simplex. Similarly,
the $i$-degeneracy of $((c,f,\alpha)_n,(x,u,\phi)_n)$ is given by
first taking the $i$-degeneracy of $(c,f,\alpha)_n$ in the
simplicial set $\W\ner\ner C$ and secondly by repeating the object
$x_i$, inserting $i+1$ times the identity 1-cell $1_{x_i}:x_i\to
x_i$, $i$ times the identity 2-cell $1_{1_{x_i}}:1_{x_i}\Rightarrow
1_{x_i}$ and, for each $i<m$, by repeating  the 1-cell
$u_m^{i}:f^i_{m*}x_{i-1}\to x_i$ and inserting the identity 2-cell
$1_{u_m^i}:u_m^i\Rightarrow u_m^i$.

Finally, observe that any $n$-simplex
$((c,f,\alpha)_n,(x,u,\phi)_n)$ of  $\W([p]\mapsto
\W E_{p,\bullet,\bullet})$, such as \eqref{twotri2}, identifies
with the $n$-simplex $((c,x),(f,u),(\alpha,\phi))_n \in \W\ner\ner
\int_{C}\D$,
$$
\xymatrix @C=6pt {(c_0,x_0)\ar[rr]^{(f^1_0,u_1^0)} && (c_1,x_1)
 \ar@/^1.5pc/[rr]^{(f^0_2,u_2^0)} \ar@/_1.5pc/[rr]_-{(f^1_2,u_2^1)}
& {\Downarrow \scriptstyle{(\alpha^1_2,\phi_2^1)}} & \hspace{-6pt}
(c_2,x_2) \ar[rrr]
\ar@/^2pc/[rrr]^{(f^0_3,u_3^0)}_(.4)*+{\Downarrow\scriptstyle{(\alpha^1_3,\phi_3^1)}}
\ar@/_2pc/[rrr]_-{(f^2_3,u_3^2)}^(.4)*+{\Downarrow \scriptstyle{(\alpha^2_3,\phi_3^2)}}& &&
(x_3, c_3) &\hspace{-4pt}\cdots&\hspace{-4pt}(x_{n-1},c_{n-1})
\ar@/^1pc/[rrr]
\ar@{}[rrr]|{\vdots}
 \ar@/_1pc/[rrr]
\ar@/^2.7pc/[rrr]^{(f^0_n,u_n^0)}\ar@{}@<20pt>[rr]|(.8){\Downarrow\scriptstyle{(\alpha^1_n,\phi_n^1)}}
\ar@/_2.7pc/[rrr]_-{(f^{n-1}_n,u_n^{n-1})}
\ar@{}@<-18pt>[rr]|(.8){\Downarrow\scriptstyle{(\!\alpha^{n-1}_n\!\!\!,\phi_n^{n-1}\!)}}&
&& (c_n,x_n),
}
$$
given by the objects $(c_m,x_m)$ of $\int_C\D$, $0\leq m\leq n$, the
1-cells $$(f^k_m,u_m^k):(c_{m-1},x_{m-1})\to (c_m,x_{m}),
\hspace{0.4cm}(0\leq k <m\leq n),$$ and 2-cells
$$(\alpha^k_m,\phi_m^k):(f^{k-1}_m,u_m^{k-1})\Rightarrow
(f^k_m,u_m^{k})\hspace{0.4cm}(0<k<m\leq n).$$ Thus, we have the simplicial isomorphism
\begin{equation}\label{114}
\xymatrix{\W([p]\mapsto
\W E_{p,\bullet,\bullet})\cong\W\ner\ner\int_{C}\D,}
\end{equation}
and therefore the homotopy equivalence
$$
\xymatrix{\thb \text{hocolim}_C\D \thb \overset{\eqref{113}\eqref{114}}\simeq |\W\ner\ner\int_{C}\D|.}
$$
 Since, for the 2-category $\int_{\!B}\D$, we have the
natural homotopy equivalence
$$\xymatrix{\tb \int_{\!B}\D\tb=|\Diag \ner\ner \int_{\!B}\D|\overset{\eqref{diwb}}\simeq |\W \ner\ner\int_{\!B}\D|,}$$
the proof is complete.
\end{proof}
After Theorem \ref{hct}, for any 2-functor $\D:C\to
\underline{\mathbf{2}\Cat}$ (or $\D:C^{\mathrm{op}}\to
\underline{\mathbf{2}\Cat}$), {\em we do not distinguish between
the classifying spaces of the simplicial $2$-category $\mathrm{hocolim}_C\D$ and of the
$2$-category $\int_C\D$}, since both represent the same homotopy
type in a natural way.

\section{The  homotopy-fibre 2-functors}  In this section, we mainly review
some necessary results concerning the more striking examples of
2-diagrams of 2-categories: the 2-diagrams of homotopy-fibre
2-categories of a 2-functor.

 As usual, if $\D:C\to \underline{\mathbf{2}\Cat}$ (resp. $\D:C^{\mathrm{op}} \to \underline{\mathbf{2}\Cat}$) and
  $F:A\to C$ are 2-functors,  we denote
by $$F^*\D=\D\,F:A\to \underline{\mathbf{2}\Cat}$$ (resp.
$F^*\D=\D\,F:A^{\mathrm{op}} \to \underline{\mathbf{2}\Cat}$) the
2-functor obtained by composing $\D$ with $F$.

Let $F:A\to C$ be any given 2-functor. Then, for any object $c\in
C$, the {\em homotopy-fibre of $F$ over $c$}
~\cite{Gray1969,Cegarra-2011}, denoted by $F\!\downarrow \!c$, is the
2-category obtained by applying the Grothendieck construction on the
2-functor $F^*C(-,c):A^{\mathrm{op}}\to \underline{\Cat}$, where
$C(-,c):C^{\mathrm{op}}\to \underline{\Cat}$ is the hom 2-functor, that is,
$$
\xymatrix{F\!\downarrow \!c \,=\int_AF^*C(-,c).  }
$$
 Thus, $F\!\downarrow \!c$
has objects the pairs $(a,p)$, with $a$ a 0-cell of $A$ and $p:Fa\to
c$ a 1-cell of $C$. A 1-cell $(u,\phi):(a,p)\to (a',p')$ consists of
a $1$-cell $u:a\to a'$ in $A$, together with a $2$-cell
$\phi:p\Rightarrow p'\circ Fu$ in the 2-category $C$,
$$
\xymatrix@C=10pt@R=14pt{ Fa\ar[rr]^{Fu}\ar[rd]_{p}&
\ar@{}[d]|(.27){\phi}|(.48){\Rightarrow}&Fa'\ar[ld]\ar@{}@<-2pt>[ld]^(.45){p'}\\
 &c& }
$$
and, for $(u,\phi),(u',\phi'):(a,p)\to (a',p')$, a 2-cell
$\alpha:(u,\phi)\Rightarrow (u',\phi')$ is a  2-cell
$\alpha:u\Rightarrow u'$ in $A$ such that $(1_{p'}\circ
F\alpha)\cdot \phi=\phi'$. Compositions and identities are given canonically.

Any 1-cell $h:c\to c'$ in $C$ gives rise to a 2-functor
$$h_*:F\!\downarrow \!c \to F\!\downarrow \!c',$$
which acts on cells by
$$
 \xymatrix @C=8pt{(a,p) \ar@/^1pc/[rr]^{\ (u,\phi)} \ar@/_1pc/[rr]_-{\ (u',\phi')}
\ar@{}[rr]|{\Downarrow\! \alpha}&  & (a',p')}
\xymatrix{\ar@{|->}[r]^{h_*}&}
 \xymatrix @C=1pc{(a,h\circ p)\ \ar@/^1pc/[rr]^{\ (u,1_h\circ \phi)}
\ar@/_1pc/[rr]_-{\ (u',1_h\circ \phi')}\ar@{}[rr]|{\Downarrow\! \alpha} &  & (a',h\circ p')},
$$
and, for $h,h':c\to c'$, any 2-cell $\psi:h\Rightarrow h'$ in $C$
produces a 2-transformation $$\psi_*:h_*\Rightarrow h'_*,$$ whose
component at any object $(a,p)$ of $F\!\downarrow \!c$ is the 1-cell
of $F\!\downarrow \!c'$
$$
\psi_*(a,p)= (1_a,\psi\circ 1_p):(a,h\circ p)\to (a,h'\circ p).
$$

In this way, we have the {\em homotopy-fibre $2$-functor}
\begin{align}\nonumber
F\!\downarrow \!-\ :&\xymatrix@R=0pt@C=35pt{C\ar[r]^-{\mathcal{Y}}&\underline{\Cat}^{C^{\mathrm{op}}}
\ar[r]^{F^*}&\underline{\Cat}^{A^{\mathrm{op}}}\ar[r]^-{\int_A-}&\underline{\mathbf{2}\Cat}},\\[-2pt] \nonumber
&\xymatrix@C=30pt{c\ar@{|->}[r]&C(-,c)\ar@{|->}[r]&F^*C(-,c)\ar@{|->}[r]&
\int_AF^*C(-,c)=F\!\downarrow \!c}
\end{align}
where $\mathcal{Y}$ is the 2-categorical Yoneda embedding; and,
quite similarly, we also have  the {\em homotopy-fibre $2$-functor}
\begin{align}\nonumber
-\!\downarrow \!F\ :&\xymatrix@R=0pt@C=35pt{C^{\mathrm{op}}
\ar[r]^-{\mathcal{Y}}&\underline{\Cat}^{C}
\ar[r]^{F^*}&\underline{\Cat}^{A}\ar[r]^-{\int_A-}&\underline{\mathbf{2}\Cat}},\\[-2pt] \nonumber
&\xymatrix@C=30pt{c\ar@{|->}[r]&C(c,-)\ar@{|->}[r]&F^*C(c,-)\ar@{|->}[r]&
\int_AF^*C(c,-)=c\!\downarrow \!F}
\end{align}
which assigns to each object $c$ of $C$ the {\em homotopy-fibre
$2$-category of $F$ under $c$}, $c\!\downarrow \!F$, whose objects
are pairs $(a,c\overset{p}\to Fa)$. The 1-cells
$(u,\phi):(a,p)\to(a',p')$ are pairs where $u:a\to a'$ is a 1-cell
of $A$ and $\phi:Fu\circ p\Rightarrow p'$ is a 2-cell of $C$, and a
2-cell $\alpha:(u,\phi)\Rightarrow (u',\phi')$ is a 2-cell
$\alpha:u\Rightarrow u'$ in $A$ such that $\phi'\cdot (F\alpha\circ
1_p)=\phi$.

These 2-diagrams $F\!\downarrow \!-$ and $-\!\downarrow \!F $ are
relevant for homotopy interests, mainly because the projection
2-functors
$$\xymatrix{F\!\downarrow \!c\ar[r]^{\pi}& A & c\!\downarrow \!F,\ar[l]_{\pi}}$$
both given on cells by
$$
\pi: \xymatrix @C=8pt{(a,p) \ar@/^1pc/[rr]^{\ (u,\phi)} \ar@/_1pc/[rr]_-{\ (u',\phi')}
\ar@{}[rr]|{\Downarrow \alpha}&  & (a',p')}
\xymatrix{\ar@{|->}[r]&}
 \xymatrix @C=0.6pc{a \ar@/^0.7pc/[rr]^{ u}
\ar@/_0.7pc/[rr]_-{ u'}\ar@{}[rr]|{\Downarrow \alpha} &  & a'},
$$
induce 2-functors
$$
\xymatrix{\int_C \!(F\!\downarrow \!-) \ar[r]^-{\Pi}& A&
\int_C\! (-\!\downarrow\! F)
\ar[l]_-{\Pi}}
$$
$$
\Pi:  \xymatrix @C=12pt{(c,(a,p)) \ar@/^1pc/[rr]^{(h,(u,\phi))} \ar@/_1pc/[rr]_-{(h',(u',\phi'))}
\ar@{}[rr]|{\Downarrow (\psi,\alpha)}&  &(c', (a',p'))}
\xymatrix{\ar@{|->}[r]&}
 \xymatrix @C=0.6pc{a \ar@/^0.7pc/[rr]^{ u}
\ar@/_0.7pc/[rr]_-{ u'}\ar@{}[rr]|{\Downarrow \alpha} &  & a'},
$$
and we have the proposition below.
\begin{proposition}[~\cite{CCH2014} Lemma
5.5]\label{globcov} For any $2$-functor $F:A\to C$, both
$2$-functors $\Pi$ above are weak equivalences.
\end{proposition}
\begin{proof} The 2-functor $\Pi:\int_C(F\! \downarrow\!-)\to
A$ is actually a retraction with section the 2-functor $\iota: A\to
\int_C(F\! \downarrow\!-)$ given by
$$
 \xymatrix @C=0.6pc{a \ar@/^0.7pc/[rr]^{ u}
\ar@/_0.7pc/[rr]_-{ u'}\ar@{}[rr]|{\Downarrow\! \alpha} &  & a'}
\xymatrix{\ar@{|->}[r]&}
 \xymatrix @C=14pt{(Fa,(a, 1_{Fa})) \ar@/^1.2pc/[rr]^{\ (Fu,(u,1_{Fu}))}
  \ar@/_1.2pc/[rr]_-{\ (Fu',(u',1_{Fu'}))}
\ar@{}[rr]|{\Downarrow\! (F\alpha,\alpha)}&  & (Fa',(a', 1_{Fa'}))}.
$$
It is clear that $\Pi\, \iota =1$. Furthermore, there is an oplax
transformation $\iota \Pi \Rightarrow 1$ given, on any object
$(c,(a,p))$ of $\int_C(F\! \downarrow\!-)$, by the 1-cell
$$
(p,(1_a,1_p)): (Fa,(a,1_{Fa}))\to (c,(a,p)),
$$
and whose component of naturality at any 1-cell
$(h,(u,\phi)):(c,(a,p))\to (c',(a',p'))$ is the 2-cell
$$
\xymatrix@C=90pt{(Fa,(a,1_{Fa}))\ar[r]^{(p,(1_a,1_p))}\ar[d]_{(Fu,(u,1_{Fu}))}
\ar@{}@<70pt>[d]|(.4){(\phi,1)}|(.55)\Leftarrow&
(c,(a,v))\ar[d]^{(h,(u,\phi))}\\
(Fa',(a',1_{Fa'}))\ar[r]_{(p',(1_{a'},1_{p'}))}&(c',(a',p')).
}
$$
Hence, for the induced maps by $\Pi$ and $\iota$ on classifying
spaces, we have $\tb \Pi\tb\,\tb \iota\tb =1$ and, by Fact
\ref{fact1}, $1\simeq \tb \iota\tb\,\tb \Pi\tb$. Therefore, we
conclude that the 2-functor $\Pi:\int_C(F\! \downarrow\!-)\to A$ is
a weak equivalence. The proof for $\Pi:\int_C(-\! \downarrow\!F)\to A$ is parallel.
\end{proof}

Therefore, for any 2-functor $F:A\to C$, both 2-diagrams of
homotopy-fibre 2-categories $F\! \downarrow\!-:C\to
\underline{\mathbf{2}\Cat}$ and $-\! \downarrow\!F:C^{\mathrm{op}}\to
\underline{\mathbf{2}\Cat}$  have the same homotopy type as the
``total" 2-category $A$ of the 2-functor $F$, in the standard sense
that there are homotopy equivalences $$\thb \mathrm{hocolim}_C(F\!
\downarrow\!-)\thb\simeq \tb A\tb\simeq \thb \mathrm{hocolim}_C(-\!
\downarrow\!F)\thb.$$

As a quick application, we have the  relative ``Quillen
Theorem A" by Chiche below. We will use the following notation: Given a commutative square in $\mathbf{2}\Cat$
$$
\xymatrix{
A \ar[r]^{F}
  \ar[d]_{G}
&
B \ar[d]^{H}
\\
D \ar[r]^{T}
&
C,
}
$$
for any object $d$ of $D$, let
\begin{equation}\label{fb}
\begin{array}{lll}
\bar{F}: G\!\downarrow\! d \to H\!\downarrow\! Td,
&&
\bar{F}: d\!\downarrow\! G \to Td \!\downarrow\! H,
\end{array}
\end{equation}
be the induced functors making commutative the squares
$$
\begin{array}{ll}
\xymatrix{
G\!\downarrow \! d \ar[r]^{\bar{F}}
  \ar[d]_{\pi}
&
H\!\downarrow\! Td \ar[d]^{\pi}
\\
A \ar[r]^{F}
&
B
}
&
\xymatrix{
d\!\downarrow \! G \ar[r]^{\bar{F}}
  \ar[d]_{\pi}
&
Td\!\downarrow\! H \ar[d]^{\pi}
\\
A \ar[r]^{F}
&
B.
}
\end{array}
$$
Both act on cells by
$$
\xymatrix @C=8pt{(a,p) \ar@/^1pc/[rr]^{\ (u,\phi)} \ar@/_1pc/[rr]_-{\ (u',\phi')}
\ar@{}[rr]|{\Downarrow\! \alpha}&  & (a',p')}
\xymatrix@C=12pt{\ar@{|->}[r]^{\bar{F}}&}
\xymatrix @C=10pt{(Fa,Tp) \ar@/^1pc/[rr]^{\ (Fu,T\phi)} \ar@/_1pc/[rr]_-{\ (Fu',T\phi')}
\ar@{}[rr]|{\Downarrow F\alpha}&  & (Fa',Tp').}
$$

\begin{theorem}[~\cite{Chich2015} Th\'eor\`{e}m 2.34] \label{rta}Let
$$
\xymatrix@C=20pt@R=20pt{A\ar[rr]^{F}\ar[rd]_{G}&&B\ar[ld]^{H}\\ &C&}
$$
be a commutative triangle of $2$-functors.

$(i)$ If, for any object $c\in C$, the induced functor
$
\bar{F}: G\!\downarrow \!c \longrightarrow H\!\downarrow \!c$
is a weak equivalence, then $F:A\to B$ is a weak equivalence.

\vspace{0.2cm}
$(ii)$ If, for any object $c\in C$, the induced functor
$
\bar{F}: c\!\downarrow \!G \longrightarrow c\!\downarrow \!H,
$
is a weak equivalence, then $F:A\to B$ is a weak equivalence.
\end{theorem}
\begin{proof} $(i)$ The 2-functor $F$ occurs in the commutative square
$$
\xymatrix{\int_C \!(G\!\downarrow \!-) \ar[r]^{\int_C\bar{F}} \ar[d]_{\Pi}&
 \int_C \!(H\!\downarrow \!-) \ar[d]^{\Pi}\\
A\ar[r]^F&B,
}
$$
where the vertical 2-functors $\Pi$ are weak equivalences and the
horizontal 2-functor at the top is a weak equivalence by the
Homotopy Invariance Theorem \ref{ith} (and the Homotopy Colimit
Theorem \ref{hct}), whence the result follows.
\end{proof}

In the particular case where $F=1_C$ is the identity 2-functor on
$C$, we have the {\em slice $2$-categories} $C\!\downarrow \!c$, of
objects over an object $c$, and $c\!\downarrow \!C$, of objects
under $c$. Proposition \ref{globcov} tells us that there are weak
equivalences $\int_C \!(C\!\downarrow \!-)\to C \leftarrow \int_C
\!(-\!\downarrow \!C)$, which is exactly what is expected from Corollary
\ref{cor1} because we have the lemma below.
\begin{lemma}[~\cite{B-C2003} Theorem 4.1]\label{lcont} For any $2$-category $C$ and any object $c\in
C$, the $2$-categories $c\!\downarrow \!C$ and $C\!\downarrow \!c$
are weakly contractible.
\end{lemma}
\begin{proof} Let $\mathrm{pt}\to c\!\downarrow \!C$,
$\mathrm{pt}\mapsto 1_c$, be the 2-functor from the terminal
2-category given by the object $1_c:c\to c$ of $c\!\downarrow \!C$.
Then, there is an oplax transformation from the constant composite
2-functor $c\!\downarrow \!C \to \mathrm{pt}\to c\!\downarrow \!C$
to the identity functor on $c\!\downarrow \!C$, whose component at
any object $p:c\to d$ is the 1-cell $(p,1_p):1_c\to p$, and whose
naturality component at a 1-cell $(u,\phi):p\to p'$ is $\phi$.  From
Fact  \ref{fact1}, it follows that the (constant) induced map on the
classifying space $\tb c\!\downarrow \!C\tb \to \tb \mathrm{pt}\tb =
\mathrm{pt}\to \tb c\!\downarrow \!C\tb$ is homotopic to the
identity map, whence the result.
\end{proof}

As an interesting application of Theorem \ref{rta}, we have the
2-categorical version of Quillen's Theorem A in ~\cite{Quillen1973}
below. First, let us borrow some terminology  from  Hirschhorn  in
~\cite[19.6.1]{Hirschhorn2009}:

\begin{itemize}
\item  a $2$-functor  $F:A\to C$ is called {\em homotopy left cofinal} if all the
homotopy-fibre 2-categories $F\!\downarrow \!c $, $c\in C$, are
weakly contractible\footnote{For $F$ a functor between small
categories, this condition is referred by  Cisinski  in
~\cite[3.3.3]{Cisinski2006} by saying that ``$F$ is
aspherical".},

\vspace{0.2cm}
\item  a $2$-functor  $F:A\to C$ is called {\em homotopy right cofinal} if all the
homotopy-fibre 2-categories $c\!\downarrow \!F$, $c\in C$, are
weakly contractible.
\end{itemize}

\begin{corollary}[~\cite{B-C2003} Theorem 1.2]\label{corta} Every homotopy left or right cofinal
$2$-functor between $2$-categories is a weak equivalence.
\end{corollary}
\begin{proof} Let $F:A\to C$ be a left homotopy cofinal 2-functor.  Then, for any object $c\in C$, the induced functor $\bar{F}:F\!\downarrow \!c\to C\!\downarrow \!c$ is a weak equivalence and, therefore,  due to the particular case of Theorem
\ref{rta} where $H=1_C$, $F$ is a weak equivalence.
\end{proof}

Quillen's Theorem B in ~\cite{Quillen1973} has also been generalized for 2-functors between 2-categories as below. We
shall first set some terminology. Following Dwyer, Kan, and Smith in
~~\cite[\S 6]{DKS1989} and Barwick and Kan in ~~\cite{BK2013}, we say
that:
\begin{itemize}
\item   a $2$-functor  $F:A\to C$
   has the {\em property} B$_l$ if, for any $1$-cell $h:c\to c'$
  in $C$, the induced $2$-functor $h_*:F\!\downarrow\! c \to
  F\!\downarrow\! c'$ is a weak equivalence\footnote{For $F$ a
  functor between small categories, this condition is referred
by Cisinski in ~\cite[6.4.1]{Cisinski2006} by saying that ``the
functor $F$ is locally constant".},

\vspace{0.2cm}
\item   a $2$-functor  $F:A\to C$
   has the {\em property} B$_r$ if, for any $1$-cell $h:c\to c'$
  in $C$, the induced $2$-functor $h^*:c'\!\downarrow\! F \to
  c\!\downarrow\!F$ is a weak equivalence.
\end{itemize}

For any 2-functor  $F: A\to C$, each object $c\in C$ determines two pullback squares in ${\mathbf{2}\Cat}$
\begin{equation}\label{sqchac}\begin{array}{c}
\xymatrix{F\!\downarrow \!c\ar[r]^{\bar{F}}\ar[d]_{\pi}&C\!\downarrow \!c
\ar[d]^{\pi}\\ A\ar[r]^{F}&C
} \hspace{0.8cm} \xymatrix{c\!\downarrow \!F\ar[r]^{\bar{F}}\ar[d]_{\pi}&c\!\downarrow \!C
\ar[d]^{\pi}\\ A\ar[r]^{F}&C
}
\end{array}
\end{equation}
(where $\bar{F}$ is \eqref{fb} for $G=F$ and $H=T=1_C$), and the extension of
Quillen's Theorem B for 2-functors in ~\cite[Theorem 3.2]{Cegarra-2011} tells us that the following theorem holds.
\begin{theorem}\label{th3.2}
 A $2$-functor $F : A\to C$
 has the property $\mathrm{B}_l$ (resp. $\mathrm{B}_r$) if and only if the left (resp. right)
 square in \eqref{sqchac} is a
homotopy pullback for every object $c\in C$.
\end{theorem}
Observe that every homotopy  left (resp. right) cofinal  2-functor
has the property $\mathrm{B}_l$ (resp. $\mathrm{B}_r$), and
therefore Corollary \ref{corta} is also a consequence of Theorem
\ref{th3.2}.

\section{Theorem B for 2-transformations}
In this section, we state and prove extensions of Theorem \ref{th3.2}
for 2-transformations between 2-diagrams of 2-categories.

For any 2-functor $\D:C\to \underline{\mathbf{2}\Cat}$ (or
$\D:C^{\mathrm{op}}\to \underline{\mathbf{2}\Cat}$) and any object $c\in C$,
let
$$\xymatrix{ \bar{c}:\D_{\!c}\longrightarrow \int_C\D}$$
denote the embedding 2-functor
$$
 \xymatrix @C=0.6pc{x \ar@/^0.7pc/[rr]^{ u}
\ar@/_0.7pc/[rr]_-{ u'}\ar@{}[rr]|{\Downarrow\! \alpha} &  & x'}\xymatrix{\ar@{|->}[r]&} \xymatrix @C=16pt{(c,x) \ar@/^1pc/[rr]^{ (1_c,u)} \ar@/_1pc/[rr]_-{(1_c,u')}
\ar@{}[rr]|{\Downarrow (1_{1_c},\alpha)}&  & (c,x').}
$$

If $\D,\,\E:C\to \underline{\mathbf{2}\Cat}$ are $2$-functors and
$\Gamma:\D\Rightarrow \E$ is a $2$-transformation, for any objects $c\in C$
and $y\in \E_c$, there is a canonical commutative square of
2-categories
\begin{equation}\label{tb2d}\begin{array}{l}\xymatrix{
\Gamma_{\!c}\!\downarrow\!y\ar[r]\ar[d]
& \E_c\!\downarrow\!y\ar[d]\\
\int_C\D\ar[r]^{\int_C\Gamma}
& \int_C\E,}
\end{array}\end{equation}
which, keeping the notations in squares
\eqref{sqchac},  is the composite of the squares
$$
\xymatrix{
\Gamma_{\!c}\!\downarrow\!y\ar[r]^{\overline{\Gamma}_{\!c}}\ar[d]_{\pi}
& \E_c\!\downarrow\!y\ar[d]^{\pi}\\
\D_c\ar[r]^{\Gamma_{\!c}}\ar[d]_{\bar{c}}&\E_c\ar[d]^{\bar{c}} \\
\int_C\D\ar[r]^{\int_C\Gamma}
& \int_C\E,}
$$
and we have the theorem below.

\begin{theorem}[Theorem $\mathrm{B}_l$ for 2-transformations] \label{tbgl}Let $\Gamma:\D\Rightarrow \E$ be a $2$-transformation, where
$\D,\,\E:C\to \underline{\mathbf{2}\Cat}$ are $2$-functors. The
following properties are equivalent:
\begin{enumerate}
\item[$(a)$] For any object $c$ in $C$ and any object $y$ in $\E_c$, the
 square \eqref{tb2d} is a homotopy pullback.

 \vspace{0.2cm}
\item[$(b)$] The $2$-functor $\int_C\Gamma:\int_C\D\to \int_C\E$ has
the property $\mathrm{B}_l$.

 \vspace{0.2cm}
\item[$(c)$] The two conditions below hold.
\begin{enumerate}
\item[$\mathrm{B1}_l$:] for
each object $c$ of $C$, the $2$-functor $\Gamma_c:\D_c\to
\E_c$ has the property $\mathrm{B}_l$.
 \vspace{0.1cm}

\item[$\mathrm{B2}_l$:] for each $1$-cell $h:c\to c'$ in $C$ and any object
$y\in \E_c$, the induced $2$-functor
$$\xymatrix@C=14pt{\overline{h}_*:\Gamma_{\!c}\!\downarrow\!y\ar[r]&
\Gamma_{\!c'}\!\downarrow\!h_*y}$$
is a weak equivalence.
\end{enumerate}
\item[$(d)$] The two conditions below hold.
\begin{enumerate}
\item[$\mathrm{B1}_l$:] for
each object $c$ of $C$, the $2$-functor $\Gamma_c:\D_c\to
\E_c$ has the property $\mathrm{B}_l$.

\vspace{0.1cm}
\item[$\mathrm{B2}'_l$:] for each $1$-cell $h:c\to c'$ in $C$, the square
\begin{equation}\label{sb'2}\begin{array}{l}\xymatrix{\D_{\!c}\ar[r]^{h_*}\ar[d]_{\Gamma_{\!c}}&\D_{\!c'}
\ar[d]^{\Gamma_{\!c'}}\\ \E_{\!c}\ar[r]^{h_*}& \E_{\!c'}}\end{array}\end{equation}
is a homotopy pullback.
\end{enumerate}
\end{enumerate}
\end{theorem}
\begin{proof}
In order to prove the result, first we  show that, for any
object $c$ in $C$ and any object $y$ in $\E_c$,  there is a  weak
equivalence
\begin{equation}\label{eqr}
\xymatrix{R:\int_C\!\Gamma\! \downarrow\!(c,y) \,\to \,\Gamma_{\!c}\!\downarrow\!y,}
\end{equation}
where  $R$ is the 2-functor acting on cells of the 2-category
$\int_C\!\Gamma\! \downarrow\!(c,y)$ in the following way: On
objects $((a,x),(p,v))$, where $a$ is an object of $C$, $x$ an
object of $\D_a$, $p:a\to c$ is a 1-cell in $C$, and
$v:p_*\Gamma_{\!a}x=\Gamma_{\!c}p_*x\to y$ is a 1-cell in $\E_c$,
$$
R((a,x),(p,v))=(p_*x,v).
$$
On 1-cells
\begin{equation}\label{1-cell}((f,u),(\alpha,\psi)):((a,x),(p,v))\to
((a',x'),(p',v')),\end{equation} where $f:a\to a'$ is in $C$,
$u:f_*x\to x'$ in $\D_a$, $\alpha:p\Rightarrow p'\circ f$ in $C$,
and $\psi:v\Rightarrow v'\circ p'_*\Gamma_{\!a'}u\circ
\alpha_*\Gamma_{\!a}x$ in $\E_c$,
$$
R((f,u),(\alpha,\psi))=(p'_*u\circ \alpha_*x,\psi):(p_*x,v)\to (p'_*x',v').
$$
And, for a 2-cell $(\beta,\phi):((f,u),(\alpha,\psi))\Rightarrow
((f',u'),(\alpha',\psi'))$, where $\beta:f\Rightarrow f'$ is in $C$
and $\phi:u\Rightarrow u'\circ \beta_*x$ in $\D_a$, satisfying the
corresponding conditions,
$$ R(\beta,\phi)=p'_*\phi\circ
1_{\alpha_*x}:(p'_*u\circ \alpha_*x,\psi)\Rightarrow (p'_*u'\circ
\alpha'_*x,\psi').
$$

This 2-functor $R$ is actually a retraction, with a section given by the induced 2-functor on homotopy-fibre 2-categories
\begin{equation}\label{eql}
\xymatrix{\bar{c}:\Gamma_{\!c}\!\downarrow\!y \,\longrightarrow \,\int_C\!\Gamma\! \downarrow\!(c,y)}
\end{equation}
 making  the diagram below commutative:
$$
\xymatrix{\Gamma_{\!c}\!\downarrow\!y\ar[r]^{\pi}\ar[d]_{\bar{c}}&\D_c \ar[d]_{\bar{c}}\ar[r]^{\Gamma_c}&\E_c\ar[d]^{\bar{c}}\\
\int_C\!\Gamma\! \downarrow\!(c,y)\ar[r]^{\pi}&\int_C\D\ar[r]^{\int_C\Gamma}&\int_C\E.
}
$$
Explicitly, $\bar{c}$ in \eqref{eql} acts on objects $(x,v)$, where $x$ is
an object of $\D_c$ and $v:\Gamma_{\!c}x\to y$ is a 1-cell of
$\E_c$, by
$$
\bar{c}(x,v)=((c,x),(1_c,v)),
$$
on 1-cells $(u,\psi):(x,v)\to (x',v')$, where $u:x\to x'$ is in
$\D_c$ and $\psi:v\Rightarrow v'\circ \Gamma_{\!c}u$ in $\E_c$,  by
$$\bar{c}(u,\psi)=((1_c,u),(1_{1_c},\psi)):((c,x),(1_c,v))\to ((c,x'),(1_c,v'))
,$$ and, on a 2-cell $\phi:(u,\psi)\Rightarrow (u',\psi')$,
$$\bar{c}(\phi)=(1_{1_c},\phi):((1_c,u),(1_{1_c},\psi))\Rightarrow ((1_c,u'),(1_{1_c},\psi')).$$

It is plain to see that $R\,\bar{c}=1$. Furthermore, there is an
oplax transformation $1\Rightarrow \bar{c}\,R$ given, on any
object $((a,x),(p,v))$ of $\int_C\!\Gamma\! \downarrow\!(c,y)$, by the 1-cell
$$
((p,1_{p_*x}),(1_p,1_v)): ((a,x),(p,v))\to ((c,p_*x),(1_c,v)),
$$
and whose naturality component at any 1-cell as in \eqref{1-cell}
is
$$
\xymatrix@C=90pt{((a,x),(p,v))\ar[r]^{((f,u),(\alpha,\psi))}\ar[d]_{((p,1_{p_*x}),(1_p,1_v))}
\ar@{}@<90pt>[d]|(.4){(\alpha,1)}|(.55){\Rightarrow}&
((a',x'),(p',v'))\ar[d]^{((p',1_{p'_*x'}),(1_{p'},1_{v'}))}\\
((c,p_*x),(1_c,v))\ar[r]_{((1_c,p'_*u\circ \alpha_*x),(1_{1_c},\psi))}&((c,p¡_*x'),(1_c,v')).
}
$$
Hence, for the maps induced by $R$ and $\bar{c}$ on classifying
spaces, we have $\tb R\tb\,\tb \bar{c}\tb =1$ and, by Fact
\ref{fact1}, $1\simeq \tb \bar{c}\tb\,\tb R\tb$. Thus, it follows
that both 2-functors $R$ and $\bar{c}$ are weak equivalences.

Let us now observe that the square \eqref{tb2d} is the composite of
the squares
$$
\xymatrix{
\Gamma_{\!c}\!\downarrow\!y \ar[r]^{\overline{\Gamma}_{\!c}}\ar[d]_{\bar{c}}
& \E_c\!\downarrow\!y\ar[d]^{\bar{c}}\\
\int_C\!\Gamma\!\downarrow\!(c,y)
\ar@{}[rd]|{(I)}
\ar[r]^{\overline{\int_C\Gamma}}\ar[d]_{\pi}&\int_C\!\E\!\downarrow\!(c,y)
\ar[d]^{\pi} \\
\int_C\D\ar[r]_{\int_C\Gamma}
& \int_C\E.}
$$
Therefore, as both vertical 2-functors $\bar{c}$ are weak equivalences,
the square \eqref{tb2d} is a homotopy pullback if and only if the
square $(I)$ above is as well. Thus, by Theorem \ref{th3.2}, it
follows that $(a)\Leftrightarrow (b)$.

To prove $(b)\Leftrightarrow (c)$, let us observe that, for any
1-cell $(h,w):(c,y)\to (c',y')$ in $\int_C\E$, there is a commutative
diagram of 2-functors
$$
\xymatrix{\int_C\!\Gamma\!\downarrow\!(c,y)\ar[rr]^{\overline{(h,w)}_*}
\ar[d]_{R} && \int_C\!\Gamma\!\downarrow\!(c',y')\ar[d]^{R}\\
\Gamma_{\!c}\!\downarrow\!y\ar[r]^{\bar{h}_*}&\Gamma_{\!c'}\!\downarrow\!h_*y\ar[r]^{w_*}&
\Gamma_{\!c'}\!\downarrow\!y',
}
$$
where both vertical 2-functors $R$ are weak equivalences. If the
2-transformation $\Gamma$ has the properties $\mathrm{B1}_l$ and
$\mathrm{B2}_l$, then both 2-functors $\bar{h}_*$ and $w_*$ in the
bottom of the diagram above are weak equivalences, and therefore the
2-functor $\overline{(h,w)}_*$ at the top is also a weak
equivalence. That is, the 2-functor $\int_C\Gamma$ has the property
$\mathrm{B}_l$.

Conversely, assume that $\int_C\Gamma$ has the property $\mathrm{B}_l$. Then,
for any object $c$ of $C$ and any 1-cell $w:y\to y'$ of $\E_c$, the
above commutative square for the case where $h=1_c$ proves that the
2-functor $w_*:\Gamma_{\!c}\!\downarrow\!y\to
\Gamma_{\!c}\!\downarrow\!y'$ is a weak equivalence; that is,
$\Gamma_{\!c}:\D_c\to\E_c$ has the property $\mathrm{B}_l$. Similarly, the
commutativity of the above square for $w=1_y$ implies that, for
every 1-cell $h:c\to c'$ on $C$ and any object $y\in \E_c$, the
2-functor $\bar{h}_*:\Gamma_{\!c}\!\downarrow\!y\to
\Gamma_{\!c'}\!\downarrow\!h_*y$ is a weak equivalence.

Finally, the equivalence $(c)\Leftrightarrow(d)$ is consequence of
the homotopy fibre characterization of homotopy pullbacks of spaces
(hence of 2-categories, see Section \ref{1.1}): For any 1-cell
$h:c\to c'$ in $C$ and any object $y\in \E_c$, we have  the equality
of composite squares $(I)+(II)=(III)+(IV)$, where
$$
\xymatrix{\Gamma_{\!c}\!\downarrow\!y\ar[r]^{\pi}\ar[d]_{\bar{\Gamma}_{\!c}}\ar@{}[rd]|{(I)}&\D_{\!c}\ar@{}[rd]|{(II)}\ar[r]^{h_*}\ar[d]_{{\Gamma}_{\!c}}&
\D_{\!c'}\ar[d]^{{\Gamma}_{\!c'}}& \ar@{}[d]|{\textstyle =}& \Gamma_{\!c}\!\downarrow\!y\ar@{}[rd]|{(III)}\ar[r]^-{\bar{h}_*}\ar[d]_{\bar{\Gamma}_{\!c}}&\Gamma_{\!c'}\!\downarrow\!h_*y\ar@{}[rd]|{(IV)}\ar[r]^{\pi}\ar[d]^{\bar{\Gamma}_{\!c'}}&
\D_{\!c'}\ar[d]^{{\Gamma}_{\!c'}}\\
\E_{\!c}\!\downarrow\!y\ar[r]^{\pi}&\E_{\!c}\ar[r]^{h_*}&\E_{\!c'}&& \E_{\!c}\!\downarrow\!y\ar[r]^-{\bar{h}_*}&\E_{\!c'}\!\downarrow\!h_*y\ar[r]^{\pi}&\E_{\!c'}.}
$$

Under the hypothesis $\mathrm{B1}_l$, the squares $(I)$  and $(IV)$
are both homotopy pullbacks (where, recall, the comma 2-categories
$\E_{\!c}\!\downarrow\!y$ and $\E_{\!c'}\!\downarrow\!h_*y$ are
weakly contractible).  Then, as the square $(II)=\eqref{sb'2}$ is a
homotopy pullback if and only if, for any object $y\in \E_y$, the
square $(I)+(II)=(III)+(IV)$ is a homotopy pullback, we conclude
that the square $\eqref{sb'2}$ is a homotopy pullback if and only if
the square $(III)$ is as well, which holds if and only if the
2-functor $\bar{h}_*:\Gamma_{\!c}\!\downarrow\!y\to
\Gamma_{\!c'}\!\downarrow\!h_*y$ is a weak equivalence.\end{proof}

Similarly, if $\D,\,\E:C^{\mathrm{op}}\to \underline{\mathbf{2}\Cat}$ are
$2$-functors and $\Gamma:\D\Rightarrow \E$ is a $2$-transformation, for any
objects $c\in C$ and $y\in \E_{c}$, there is a commutative  square
\begin{equation}\label{tb2do}\begin{array}{l}\xymatrix{
y\!\downarrow\!\Gamma_{\!c}\ar[r]\ar[d]
& y\!\downarrow\!\E_c\ar[d]\\
\int_C\D\ar[r]^{\int_C\Gamma}
& \int_C\E,}
\end{array}\end{equation}
defined as the composite of the squares
$$
\xymatrix{
y\!\downarrow\!\Gamma_{\!c}\ar[r]^{\overline{\Gamma}_{\!c}}\ar[d]_{\pi}
& y\!\downarrow\!\E_c\ar[d]^{\pi}\\
\D_c\ar[r]^{\Gamma_{\!c}}\ar[d]_{\bar{c}}&\E_c\ar[d]^{\bar{c}} \\
\int_C\D\ar[r]^{\int_C\Gamma}
& \int_C\E,}
$$
and we have the theorem below.
\begin{theorem}[Theorem $\mathrm{B}_r$ for 2-transformations] \label{tbgo}
Let $\Gamma:\D\Rightarrow\E$ be a $2$-transformation, where
$\D,\,\E:C^{\mathrm{op}}\to \underline{\mathbf{2}\Cat}$ are $2$-functors. The
following properties are equivalent:
\begin{enumerate}
\item[$(a)$] For any object $c$ in $C$ and any object $y$ in $\E_c$, the
 square \eqref{tb2do} is a homotopy pullback.

  \vspace{0.2cm}
\item[$(b)$] The $2$-functor $\int_C\Gamma:\int_C\D\to \int_C\E$ has
the property $\mathrm{B}_r$.

 \vspace{0.2cm}
\item[$(c)$] The two conditions below hold.
\begin{enumerate}
\item[$\mathrm{B1}_r$:] for
each object $c$ of $C$, the $2$-functor $\Gamma_c:\D_c\to
\E_c$ has the property $\mathrm{B}_r$.

 \vspace{0.1cm}

\item[$\mathrm{B2}_r$:] for each $1$-cell $h:c\to c'$ in $C$ and any object
$y'\in \E_{c'}$, the induced $2$-functor
$$\xymatrix@C=14pt{\bar{h}^*:\, y'\!\downarrow\!\Gamma_{\!c'}\ar[r]&
h^*y'\!\downarrow\!\Gamma_{\!c}}$$
is a weak equivalence.
\end{enumerate}
\item[$(d)$] The two conditions below hold.
\begin{enumerate}
\item[$\mathrm{B1}_r$:] for
each object $c$ of $C$, the $2$-functor $\Gamma_c:\D_c\to
\E_c$ has the property $\mathrm{B}_r$.

\vspace{0.1cm}
\item[$\mathrm{B2}'_r$:] for each $1$-cell $h:c\to c'$ in $C$, the square
$$\xymatrix{\D_{\!c'}\ar[r]^{h_*}\ar[d]_{\Gamma_{\!c'}}&\D_{\!c}
\ar[d]^{\Gamma_{\!c}}\\ \E_{\!c'}\ar[r]^{h_*}&
\E_{\!c}}$$ is a homotopy pullback.
\end{enumerate}
\end{enumerate}
\end{theorem}
\begin{proof} This is parallel to the proof of Theorem \ref{tbgl} given above,
and we leave it to the reader. We simply note that, in this case, the
weak equivalence
\begin{equation}\label{eqro}
\xymatrix{R:(c,y)\! \downarrow\!\int_C\!\Gamma \,\to \,y\!\downarrow\!\Gamma_{\!c}},
\end{equation}
for each objects $c\in C$ and $y\in \E_c$, is defined as below.

 On objects $((x,a),(p,v))\in (c,y)\!
\downarrow\!\int_C\!\Gamma$, where $a$ is an object of $C$, $x$ an
object of $\D_a$, $p:c\to a$ is a 1-cell in $C$ and $v:y\to
p^*\Gamma_{\!a}x=\Gamma_{\!c}p^*x$ is a 1-cell in $\E_c$,
$$
R((a,x),(p,v))=(p^*x,v).
$$
On 1-cells $((f,u),(\alpha,\psi)):((a,x),(p,v))\to
((a',x'),(p',v'))$  where $f:a\to a'$ is in $\C$, $u:x\to f^*x'$ in
$\D_a$, $\alpha: f\circ p\Rightarrow\circ p'$ in $C$, and $\psi:
 \alpha^*\Gamma_{\!a'}x'\circ p^*\Gamma_{\!a}u\circ v \Rightarrow v'$ in
$\E_c$,
$$
R((f,u),(\alpha,\psi))=(\alpha^*x'\circ p^*u,\psi):(p^*x,v)\to (p'^*x',v'),
$$
and, for a 2-cell $(\beta,\phi):((f,u),(\alpha,\psi))\Rightarrow
((f',u'),(\alpha',\psi'))$, where $\beta:f\Rightarrow f'$ is in $C$
and $\phi:\beta^*x'\circ u\Rightarrow u'$ in $\D_a$,
$$ R(\beta,\phi)=
1_{\alpha'^*x'}\circ p^*\phi:(\alpha^*x'\circ p^*u,\psi)\Rightarrow (
\alpha'^*x'\circ p^*u',\psi').
$$
\end{proof}

Observe that, in the particular case where $C=\mathrm{pt}$ the terminal
2-category, Theorems \ref{tbgl} and \ref{tbgo} state exactly the same as Theorem
\ref{th3.2}.

Furthermore,  in the specific case where $\E=\mathrm{pt}$ is the constant terminal 2-category, for
any 2-functor $\D:C\to \underline{\mathbf{2}\Cat}$ or $\D:C^{\mathrm{op}}\to
\underline{\mathbf{2}\Cat}$, the projection 2-functors $\pi$ are actually
isomorphisms $\Gamma_{\!c}\!\downarrow\!\mathrm{pt}\cong
\D_{\!c}\cong \mathrm{pt}\!\downarrow\!\Gamma_{\!c}$, and Theorems
\ref{tbgl} and \ref{tbgo} give as a corollary the following
2-categorical version of the relevant Quillen's detection
principle for homotopy pullback diagrams ~\cite[Lemma in p.
14]{Quillen1973} (see ~\cite[Theorem 4.3]{CCH2014}  for a more
general bicategorical result). Let us also stress that  the weak
equivalences \eqref{eqr} and \eqref{eqro}, in this case where
$\E=\mathrm{pt}$, establish weak equivalences
$$
\xymatrix@C=16pt{\pi\!\downarrow\!c \ar[r]\ar@{}@<-2pt>[r]^{\sim}& \D_{\!c}&c\!\downarrow\!\pi \ar[l]\ar@{}@<2pt>[l]_{\sim}}
$$
between the homotopy-fibre 2-categories of the projection
$2$-functor $\pi:\int_C\D \to C$ over the objects of $C$ and the
2-categories attached by the 2-diagram to these objects.

\begin{corollary}[Detecting homotopy pullbacks]\label{hpdia} Let $C$ be a $2$-category.
For any $2$-functor $\D:C\to \underline{\mathbf{2}\Cat}$ (resp. $\D:C^{\mathrm{op}}\to \underline{\mathbf{2}\Cat}$), the
following statements are equivalent:
\begin{enumerate}
\item[$(a)$]
 for any object $c$ in $C$ the square
\begin{equation}\label{sqfibl} \begin{array}{l}
\xymatrix{\D_{\!c}\ar[r]\ar[r]\ar[d]_{\bar{c}} &\mathrm{pt}\ar[d]^{c}  \\
\int_C\D\ar[r]^{\pi}&C}  \end{array}
\end{equation}
is a homotopy pullback.

\vspace{0.2cm}
\item[$(b)$]  the projection $2$-functor $\pi:\int_C\D\to C$  has
the property $\mathrm{B}_l$ (resp. $\mathrm{B}_r$) .

\vspace{0.2cm}
\item[$(c)$]  for each $1$-cell $h:c\to c'$ of $C$, the $2$-functor
$h_*:\D_{\!c}\to \D_{\!c'}$ (resp. $h^*:\D_{\!c'}\to \D_{\!c}$)
is a weak equivalence.
\end{enumerate}
\end{corollary}

The following consequence of Theorems \ref{tbgl} and \ref{tbgo} is closely related to Theorem \ref{ith} and Corollary \ref{corta}.

\begin{corollary}[Theorem A for 2-diagrams] Let $\Gamma:\D\Rightarrow \E$ be a $2$-transformation, where
$\D,\,\E:C\to \underline{\mathbf{2}\Cat}$ (resp. $\D,\,\E:C^{\mathrm{op}}\to
\underline{\mathbf{2}\Cat}$) are $2$-functors. The following
statements are equivalent:
\begin{enumerate}
\item[$(a)$] the $2$-functor $\int_C\Gamma:\int_C\D\to \int_C\E$ is homotopy
left (resp. right) cofinal.

\vspace{0.2cm}

\item[$(b)$] for any object $c\in C$, the $2$-functor $\Gamma_{\!c}:\D_c\to
\E_c$ is homotopy left (resp. right)  cofinal.
\end{enumerate}
\end{corollary}

\section{Changing the indexing 2-category}

If $F:A\to C$ is a 2-functor between 2-categories, then any
2-functor $\D:C\to \underline{\mathbf{2}\Cat}$, or
$\D:C^{\mathrm{op}}\to \underline{\mathbf{2}\Cat}$, gives rise to a
pullback of 2-categories
\begin{equation}\label{sqchb}\begin{array}{c}
\xymatrix{\int_AF^*\D\ar[r]^{\bar{F}}\ar[d]_{\pi}&\int_C\D\ar[d]^{\pi}\\
A\ar[r]^{F}&C
}
\end{array}
\end{equation}
where $\bar{F}$ is given by
\begin{equation}\label{ilf}
\bar{F}:\xymatrix@C=30pt{(a,x)\ar@{}[r]|{\Downarrow (\alpha,\phi)}\ar@/^1pc/[r]^{(f,u)}\ar@/_1pc/[r]_{(g,v)}&(a',x')
} \mapsto  \xymatrix@C=30pt{(Fa,x)
\ar@{}[r]|{\Downarrow (F\alpha,\phi)}\ar@/^1pc/[r]^{(Ff,u)}\ar@/_1pc/[r]_{(Fg,v)}&(Fa',x').
 }
\end{equation}

Our first result here completes Corollary \ref{hpdia}:

\begin{theorem}\label{qlemm} Let $C$ be a $2$-category and  $\D:C\to
\underline{\mathbf{2}\Cat}$ (resp. $\D:C^{\mathrm{op}}\to
\underline{\mathbf{2}\Cat}$) a $2$-functor. The following statements
are equivalent:

\begin{enumerate}
\item[$(a)$] For any $2$-functor $F:A\to C$, the  square \eqref{sqchb} is a
homotopy pullback.

 \vspace{0.2cm}
\item[$(b)$] For any 1-cell $h:c\to c'$ in $C$, the
$2$-functor $h_*:\D_{\!c}\to \D_{\!c'}$ (resp. $h^*:\D_{\!c'}\to
\D_{\!c}$) is a weak equivalence.
\end{enumerate}
\end{theorem}
\begin{proof} Suppose $(a)$ holds. Let
$c:\mathrm{ pt}\to C$ be the 2-functor given by any object $c\in C$. As we have quite an obvious isomorphism
$\int_{\mathrm{pt}}c^*\D\cong \D_{\!c}$, the square
\begin{equation}\label{rrsq}\begin{array}{c}
\xymatrix{\D_{\!c}\ar[r]^{\bar{c}}\ar[d]&\int_C\D\ar[d]^{\pi}\\
\mathrm{pt}\ar[r]^{c}&C
}\end{array}
\end{equation}
is, by hypothesis, a homotopy pullback. Hence, the result follows
from Corollary \ref{hpdia}.

Conversely, assume $(b)$ holds. Again by Corollary \ref{hpdia},
for any object $c\in C$, the square \eqref{rrsq} above is a homotopy
pullback.  Since, for any given 2-functor $F:A\to C$, the 2-functor
$F^*\!\D:A\to \underline{\mathbf{2}\Cat}$ (resp.
$F^*\!\D:A^{\mathrm{op}}\to \underline{\mathbf{2}\Cat}$) trivially
is under the same hypothesis $(b)$ as $\D$, it follows that,
for any object $a\in A$, both the left side and the composite square
in the diagram
$$
\xymatrix{\D_{F\!a}\ar[r]^-{\bar{a}}\ar[d]&\int_AF^*\!\D\ar[d]^{\pi}\ar[r]^{\bar{F}}
&\int_C\D\ar[d]^{\pi}\\
\mathrm{ pt}\ar[r]^{a}&A\ar[r]^{F}&C}
$$
are homotopy pullbacks. Then, from  the homotopy fibre
characterization of homotopy pullbacks, it follows that the right
side square above is a homotopy pullback, as required.
\end{proof}
Next, we state the complementary counterpart to the theorem above.
If $F:A\to C$ and
 $\D:C^{\mathrm{op}}\to \underline{\mathbf{2}\Cat}$ are 2-functors, for any
 objects $c\in C$ and $z\in \D_{\!c}$, let
 $$
\xymatrix{j_z:\, F\!\downarrow\!c\to \int_AF^*\D}
 $$
be the 2-functor defined on cells by
\begin{equation}\label{jz}
 \xymatrix@C=10pt{(a,p) \ar@/^0.8pc/[rr]^{(u,\phi)}
\ar@/_0.8pc/[rr]_-{(u',\phi')}\ar@{}[rr]|{\Downarrow\! \alpha} &  &
(a',p')}\xymatrix@C=12pt{\ar@{|->}[r]^{j_z}&} \xymatrix@C=16pt{(a,p^*z)
 \ar@/^1pc/[rr]^{
(u,\phi^*z)} \ar@/_1pc/[rr]_-{(u',\phi'^*z)} \ar@{}[rr]|{\Downarrow
(\alpha,1_{\phi'^*\!z})}&  & (a',p'^*z).}
\end{equation}

\begin{theorem}\label{tch2l}

For a $2$-functor $F:A\to C$, the following statements are
equivalent:
\begin{enumerate}
\item[$(a)$] $F:A\to C$ has the property $\mathrm{B}_l$.

\vspace{0.2cm}
\item[$(b)$] For any $2$-functor $\D:C^{\mathrm{op}}\to \underline{\mathbf{2}\Cat}$, the $2$-functor
$\bar{F}:\int_AF^*\D \to \int_C\D$  has the property
$\mathrm{B}_l$.

\vspace{0.2cm}
\item[$(c)$] For any $2$-functor $\D:C^{\mathrm{op}}\to
\underline{\mathbf{2}\Cat}$, and any objects $c\in C$ and $z\in
\D_{\!c}$, the commutative square
\begin{equation}\label{jysq}\begin{array}{c}
\xymatrix{F\!\downarrow\!c\ar[r]^{\bar{F}}\ar[d]_{j_z}&C\!\downarrow\!c\ar[d]^{j_z}\\
\int_AF^*\D\ar[r]^{\bar{F}}&\int_C\D
}
\end{array}
\end{equation}
is a homotopy pullback.

\vspace{0.2cm}
\item[$(d)$] For any $2$-functor $\D:C^{\mathrm{op}}\to
\underline{\mathbf{2}\Cat}$, the square \eqref{sqchb}
$$
\xymatrix{\int_AF^*\D\ar[r]^{\bar{F}}\ar[d]_{\pi}&\int_C\D\ar[d]^{\pi}\\
A\ar[r]^{F}&C}
$$
is a homotopy pullback.
\end{enumerate}
\end{theorem}
\begin{proof} For any objects $c\in C$ and $z\in \D_{\!c}$, let
\begin{equation}\label{eqbp}
\xymatrix{\bar{\pi}:\bar{F}\!\downarrow\!(c,z) \,\longrightarrow \,F\!\downarrow\!c}
\end{equation}
 be the induced 2-functor on homotopy-fibre 2-categories making the diagram below commutative:
$$
\xymatrix{\bar{F}\!\downarrow\!(c,z)\ar[r]^{\pi}\ar[d]_{\bar{\pi}}&\int_AF^*\D
\ar[d]_{\pi}\ar[r]^{\bar{F}}&\int_C\D\ar[d]^{\pi}\\
F\!\downarrow\!c\ar[r]^{\pi}&A\ar[r]^{F}&C.
}
$$
Explicitly, this $\bar{\pi}$ acts on objects $((a,x),(p,v))$, where
$a$ is an object of $A$, $x$ an object of $\D_{F\!a}$, $p:Fa\to c$
is a 1-cell in $C$ and $v:x\to p^*z$ is a 1-cell in $\D_{Fa}$, by
$$
\bar{\pi}((a,x),(p,v))=(a,p).
$$
On 1-cells
\begin{equation}\label{1-cell3}((f,u),(\phi,\beta)):((a,x),(p,v))\to
((a',x'),(p',v')),\end{equation} where $f:a\to a'$ is a 1-cell of
$A$, $u:x\to (Ff)^*x'$ of $\D_{F\!a}$, $\phi:p\Rightarrow p'\circ
Ff$ a 2-cell of $C$, and $\beta:\phi^*z\circ v\Rightarrow
(Ff)^*v'\circ u$ a 2-cell of $\D_{F\!a}$,
$$
\bar{\pi}((f,u),(\phi,\beta))=(f,\phi):(a,p)\to (a',p'),
$$
and, on a 2-cell $(\alpha,\psi):((f,u),(\phi,\beta))\Rightarrow
((f',u'),(\phi',\beta'))$, where $\alpha:f\Rightarrow f'$ is a
2-cell of $A$ and $\psi:(F\alpha)^*x'\circ u\Rightarrow u'$ a 2-cell
in  $\D_{Fa}$, satisfying the corresponding conditions,
$$ \bar{\pi}(\alpha,\psi)=\alpha:(f,\phi)\Rightarrow (f',\phi').
$$

The 2-functor $\bar{\pi}$ is actually a retraction, with a section given by 
the 2-functor
 $$
\xymatrix{i_z:\, F\!\downarrow\!c\to \bar{F}\!\downarrow\!(c,z),}
 $$
which acts on  objects by
$$
i_z(a,p)=((a,p^*z),(p,1_{p^*z})),
$$
on a 1-cell $(f,\phi):(a,p)\to (a',p')$ by
$$
i_z(f,\phi)=((f,\phi^*z),(\phi,1_{\phi^*z})):((a,p^*z),(p,1_{p^*z})) \to
((a',p'^*z),(p',1_{p'^*z})),
$$
and on a 2-cell $\alpha:(f,\phi)\Rightarrow (f',\phi')$ by
$$
i_z(\alpha)=(\alpha,1_{\phi^*z}):((f,\phi^*z),(\phi,1_{\phi^*z}))\Longrightarrow
((f',\phi'^*z),(\phi',1_{\phi'^*z})).
$$

It is clear that $\bar{\pi}\,i_z=1$. Furthermore, there is an oplax
transformation $1\Rightarrow i_z\,\bar{\pi}$ given, on any
object $((a,x),(p,v))$ of $\bar{F}\!\downarrow\!(c,z)$, by  the 1-cell
$$
((1_a,v),(1_p,1_v)): ((a,x),(p,v))\to ((a,p_*z),(p,1_{p^*z})),
$$
and whose naturality component  at any 1-cell as in
\eqref{1-cell3} is
$$
\xymatrix@C=90pt{((a,x),(p,v))\ar[r]^{((f,u),(\phi,\beta))}
\ar[d]_{((1_a,v),(1_p,1_v))}
\ar@{}@<90pt>[d]|(.4){(1_f,\beta)}|(.55){\Rightarrow}&
((a',x'),(p',v'))\ar[d]^{((1_{a'},v'),(1_{p'},1_{v'}))}\\
((a,p_*z),(p,1_{p^*z}))\ar[r]_{((f,\phi^*z),(\phi,1_{\phi^*\!z}))}&
((a',p'^*z),(p',1_{p'^*z})).
}
$$

Hence, for the maps induced by $\bar{\pi}$ and $i_z$ on classifying
spaces,  $\tb \bar{\pi}\tb\,\tb i_z\tb =1$ and, by Fact
\ref{fact1}, $1\simeq \tb i_z\tb\,\tb \bar{\pi}\tb$. Therefore,
 both 2-functors $\bar{\pi}$ and $i_z$ are weak
equivalences.

Let us now observe that the square \eqref{jysq} is the composite of
the squares
$$
\xymatrix{F\!\downarrow\!c \ar[r]^{\bar{F}}\ar[d]_{i_z}
& C\!\downarrow\!c\ar[d]^{i_z}\\
\bar{F}\!\downarrow\!(c,z)
\ar@{}[rd]|{(I)}
\ar[r]^{\bar{\bar{F}}}\ar[d]_{\pi}&\int_C\!\D\!\downarrow\!(c,z)
\ar[d]^{\pi} \\
\int_AF^*\D\ar[r]^{\bar{F}}
& \int_C\D,}
$$
where both vertical 2-functors $i_z$ are weak equivalences. It
follows that the square \eqref{jysq}  is a homotopy pullback if and
only if the square $(I)$ above is as well. As, by Theorem
\ref{th3.2}, the squares $(I)$ are homotopy pullbacks, for all
objects $(c,z)$ of $\int_C\D$, if and only if the 2-functor
$\bar{F}:\int_AF^*\D\to \int_C\D$ has the property $\mathrm{B}_l$,
the equivalence $(b)\Leftrightarrow (c)$ is proven.

The equivalence $(a)\Leftrightarrow (b)$ follows from the fact that,
for any 1-cell $(h,w):(c,z)\to (c',z')$ in $\int_C\D$, the square of
2-functors
$$
\xymatrix{\bar{F}\!\downarrow\!(c,z)\ar[d]_{\bar{\pi}}
\ar[r]^{\overline{(h,w)}_*}&\bar{F}\!\downarrow\!(c',z')\ar[d]^{\bar{\pi}}\\
F\!\downarrow\!c\ar[r]^{\bar{h}_*}&F\!\downarrow\!c'
}
$$
commutes, where, recall, both vertical 2-functors $\bar{\pi}$ are weak
equivalences. So, the 2-functors $\overline{(h,w)}_*$ at the top
are weak equivalences if and only if the 2-functors $\bar{h}_*$ at
the bottom are as well. This  directly means that $(a)\Rightarrow (b)$, and the
converse follows from taking any $\D$ such that $\D_c\neq \emptyset$ for all $c\in C$.

Next, we prove that $(c)\Rightarrow (d)$: For any object $(c,z)$ of
$\int_C\D$, we have the  squares
$$
\xymatrix{{F}\!\downarrow\!c\ar[r]^{j_z}
\ar[d]_{{\bar{F}}}&
\int_A\!F^*\D\ar[d]_{\bar{{F}}}\ar[r]^-{\pi}&
A\ar[d]^{F}\\
C\!\downarrow\!c\ar[r]^{j_z}&\int_C\D\ar[r]^{\pi}&
C,}
$$
whose composite is the left square in \eqref{sqchac}, and where
$j_z(1_c)=(c,z)$. By hypothesis, the left square is a homotopy
pullback. Furthermore, as $F$ has the property $\mathrm{B}_l$, owing
to the already proven implication $(c)\Rightarrow (a)$, Theorem
\ref{th3.2} implies that the composite square is also a homotopy
pullback. Therefore, by the homotopy fibre characterization, the right
square above is a homotopy pullback as well.

Finally,  $(d)\Rightarrow (a)$ is easy: For any object $c\in C$ take
$\D=C(-,c):C^{\mathrm{op}}\to \underline{\Cat}\subseteq
\underline{\mathbf{2}\Cat}$. Then, by hypothesis, the square
$$\xymatrix{\int_AF^*C(-,c)\ar[r]\ar[d]&\int_CC(-,c)\ar[d]&
\ar@{}@<-28pt>[d]|{\textstyle =}
 F\!\downarrow\!c\ar[r]\ar[d]& C\!\downarrow\!c\ar[d] \\
A\ar[r]^F&C&A\ar[r]^F&C }
$$
is a homotopy pullback, whence the result follows from Theorem
\ref{th3.2}.\end{proof}

Likewise, if $F:A\to C$ and
 $\D:C\to \underline{\mathbf{2}\Cat}$ are 2-functors, for any
 objects $c\in C$ and $z\in \D_{\!c}$, we have the 2-functor
 $$
\xymatrix{j_z:\, c\!\downarrow\!F\to \int_AF^*\D}
 $$
given by
$$
 \xymatrix@C=10pt{(a,p) \ar@/^0.8pc/[rr]^{(u,\phi)}
\ar@/_0.8pc/[rr]_-{(u',\phi')}\ar@{}[rr]|{\Downarrow\! \alpha} &  &
(a',p')}\xymatrix@C=12pt{\ar@{|->}[r]^{j_z}&} \xymatrix@C=16pt{(a,p_*z)
 \ar@/^1pc/[rr]^{
(u,\phi_*z)} \ar@/_1pc/[rr]_-{(u',\phi'_*\!z)} \ar@{}[rr]|{\Downarrow
(\alpha,1_{\phi'_*z})}&  & (a',p'_*z)}
$$
and the  result below holds.

\begin{theorem}\label{tch2o}
For a $2$-functor $F:A\to C$, the following statements are
equivalent:
\begin{enumerate}
\item[$(a)$] $F:A\to C$ has the property $\mathrm{B}_r$.

\vspace{0.2cm}
\item[$(b)$] For any $2$-functor $\D:C\to \underline{\mathbf{2}\Cat}$, the $2$-functor
$\bar{F}:\int_AF^*\D \to \int_C\D$  has the property
$\mathrm{B}_r$.

\vspace{0.2cm}
\item[$(c)$] For any $2$-functor $\D:C\to
\underline{\mathbf{2}\Cat}$, and any objects $c\in C$ and $z\in
\D_{\!c}$, the commutative square
$$
\xymatrix{c\!\downarrow\!F\ar[r]^{\bar{F}}\ar[d]_{j_z}&c\!\downarrow\!C\ar[d]^{j_z}\\
\int_AF^*\D\ar[r]^{\bar{F}}&\int_C\D
}
$$
is a homotopy pullback.

\vspace{0.2cm}
\item[$(d)$] For any $2$-functor $\D:C\to
\underline{\mathbf{2}\Cat}$, the square \eqref{sqchb}
$$
\xymatrix{\int_AF^*\D\ar[r]^{\bar{F}}\ar[d]_{\pi}&\int_C\D\ar[d]^{\pi}\\
A\ar[r]^{F}&C}
$$
is a homotopy pullback.
\end{enumerate}
\end{theorem}

A main consequence is the corollary below.

\begin{corollary}[Homotopy Cofinality Theorem] Let $F:A\to C$ be a $2$-functor
  between $2$-categories. The  statements below are equivalent.
\begin{enumerate}
\item[$(a)$] $F:A\to C$ is homotopy left (resp. right) cofinal.

\vspace{0.2cm}
\item[$(b)$] For any $2$-functor $\D:C^{\mathrm{op}}\to \underline{\mathbf{2}\Cat}$
(resp. $\D:C\to \underline{\mathbf{2}\Cat}$) the induced
$2$-functor $\bar{F}:\int_AF^*\D \to \int_C\D$  is homotopy left
(resp. right) cofinal.

\vspace{0.2cm}
\item[$(c)$] For any $2$-functor $\D:C^{\mathrm{op}}\to \underline{\mathbf{2}\Cat}$
  (resp. $\D:C\to \underline{\mathbf{2}\Cat}$) the induced
$2$-functor $\bar{F}:\int_AF^*\D \to \int_C\D$  is a weak
equivalence.
\end{enumerate}
\end{corollary}
\begin{proof}
The equivalence $(a)\Leftrightarrow (b)$ follows from Theorems
\ref{tch2l} and \ref{tch2o}. The implication $(b)\Rightarrow (c)$
follows from Corollary \ref{corta}. To prove the remaining
$(c)\Rightarrow (a)$, take, for any object $c$ of $\C$, the
2-functor $\D=C(-,c):C^{\mathrm{op}}\to \underline{\Cat}\subseteq
\underline{\mathbf{2}\Cat}$. Then, by hypothesis, the 2-functor
$$ \xymatrix@R=10pt{\int_AF^*C(-,c)\ar@{=}[d]\ar[r]^{\bar{F}}&
\int_CC(-,c)\ar@{=}[d]\\
F\!\downarrow\!c\ar[r]^{\bar{F}}&C\!\downarrow\!c} $$
is a weak equivalence. Therefore,  $F\!\downarrow\!c$ is weakly
contractible as
 $C\!\downarrow\!c$ is, by Lemma \ref{lcont}.
\end{proof}

Next, we show conditions on a 2-category $C$ in order for the square
\eqref{sqchb} to always be a homotopy pullback (cf. ~\cite[Theorem
3.8]{CHR2015}).

\begin{corollary}\label{corfin} Let $C$ be a $2$-category. Then, the following
properties are equivalent:
\begin{enumerate}

\vspace{0.1cm}
\item[$(i)$] For any 1-cell $h:c\to c'$ and any object $x$ of $C$,
the functor $h_*:C(x,c)\to C(x,c')$ is a weak equivalence.

\vspace{0.1cm}
\item[$(i')$] For any 1-cell $h:c\to c'$ and any object $x$ of $C$,
the functor $h^*:C(c',x)\to C(c,x)$ is a weak equivalence.

\vspace{0.1cm}
\item[$(ii)$] For any two objects $c,c'\in C$, the canonical square
$$
\xymatrix{C(c,c')\ar[r]^{in}\ar[d]& C\!\downarrow c'\ar[d]^{\pi}\\
\mathrm{pt}\ar[r]^{c}&C
}
$$
is a homotopy pullback\footnote{This implies that, for any object $c\in C$,
$C(c,c)=\Omega(C,c)$; that is,  the
category $C(c,c)$ is a loop object for the pointed
2-category $(C,c)$.}.

 \vspace{0.1cm}
\item[$(ii')$] For any two objects $c,c'\in C$, the canonical square
$$
\xymatrix{C(c,c')\ar[r]^{in}\ar[d]& c\!\downarrow C\ar[d]^{\pi}\\
\mathrm{pt}\ar[r]^{c'}&C
}
$$
is a homotopy pullback.

\vspace{0.1cm}
\item[$(iii)$] For any $2$-functor $\D:C\to
\underline{\mathbf{2}\Cat}$, the $2$-functor ${h}_*:\D_{\!c}\to
\D_{\!c'}$ induced for any 1-cell $h:c\to c'$ of $C$ is a weak
equivalence.

\vspace{0.1cm}
\item[$(iii')$] For any $2$-functor $\D:C^{\mathrm{op}}\to
\underline{\mathbf{2}\Cat}$, the  $2$-functor
${h}^*:\D_{\!c'}\to \D_{\!c}$ induced for any 1-cell $h:c\to c'$
of $C$ is a weak equivalence.

\vspace{0.1cm}

\item[$(iv)$]  For any $2$-functors $F:A\to C$ and
$\D:C\to \underline{\mathbf{2}\Cat}$, the square
$$
\xymatrix{\int_AF^*\D\ar[r]^{\bar{F}}\ar[d]_{\pi}&\int_C\D\ar[d]^{\pi}\\
A\ar[r]^{F}&C
}
$$
is a homotopy pullback.

\vspace{0.1cm}
\item[$(iv')$]  For any $2$-functors $F:A\to C$ and
$\D:C^{\mathrm{op}}\to \underline{\mathbf{2}\Cat}$, the square
$$
\xymatrix{\int_AF^*\D\ar[r]^{\bar{F}}\ar[d]_{\pi}&\int_C\D\ar[d]^{\pi}\\
A\ar[r]^{F}&C
}
$$
is a homotopy pullback.

\vspace{0.1cm}
\item[$(v)$] Any $2$-functor $F:A\to C$ has the property
$\mathrm{B}_r$.

\vspace{0.1cm}
\item[$(v')$] Any $2$-functor $F:A\to C$ has the property
$\mathrm{B}_l$.

\end{enumerate}

\end{corollary}
\begin{proof}

$(i)\Leftrightarrow (ii)$\footnote{The implication $(i)\Rightarrow
(ii)$ was proven by Del Hoyo in ~\cite[Theorem 8.5]{delHoyo2012}.}:
Let $c:\mathrm{pt}\to C$ be the $2$-functor given for any object
$c\in C$. Then, for any object $x\in C$, there is quite an obvious
natural isomorphism
$$\xymatrix{c\!\downarrow\!x=\int_{\mathrm{pt}}c^*C(-,x)\cong C(c,x)}$$
between the homotopy-fibre 2-category (actually a category) of
$c:\mathrm{pt}\to C$ over $x$ and the hom-category $C(c,x)$. Then,
 any 1-cell $h:x\to y$ in $C$ induces a weak equivalence
$h_*:c\!\downarrow\!x\to c\!\downarrow\!y$  if and only if the
induced $h_*:C(c,x)\to C(c,y)$ is a weak equivalence. It follows
that the 2-functor $c:\mathrm{pt}\to C$ has the property
$\mathrm{B}_l$ if and only if the hypothesis in $(i)$ holds. On
the other hand, by Theorem \ref{th3.2}, the 2-functor
$c:\mathrm{pt}\to C$ has the property $\mathrm{B}_l$ if and only if,
for any object $c'\in C$, the square
$$
\xymatrix{c\!\downarrow\!c'\ar[r]\ar[d]&C\!\downarrow\!c'\ar[d]&
\ar@{}@<-28pt>[d]|{\textstyle \cong} C(c,c')\ar[r]\ar[d]&
 C\!\downarrow\!c'\ar[d] \\
\mathrm{pt}\ar[r]^c&C&\mathrm{pt}\ar[r]^c&C
}
$$
is a homotopy pullback, that is, if and only if $(ii)$ holds.

$(iii)\Rightarrow (i)$: For any object $x\in C$, the result follows
by applying the
 hypothesis in $(iii)$ to the 2-functor  $\D=C(x,-):C\to \underline{\Cat}$.

$(i)\Rightarrow (v)$: by Theorem \ref{qlemm}, for any 2-functor
$F:A\to C$ and any object $c\in C$, the square
$$
\xymatrix{\int_AF^*C(c,-)\ar[r]\ar[d]&\int_CC(c,-)\ar[d]&
\ar@{}@<-28pt>[d]|{\textstyle =} c\!\downarrow\!F\ar[r]\ar[d]&
 c\!\downarrow\!C\ar[d] \\
A\ar[r]^F&C&A\ar[r]^F&C
}
$$
is a homotopy pullback. Then, $F$ has the property $\mathrm{B}_r$ by
Theorem \ref{th3.2}.

$(v)\Rightarrow (iv)$: This follows from Theorem \ref{tch2o}.

$(iv)\Rightarrow(iii)$: This follows from Theorem \ref{qlemm}.

Thus, we have $(i)\Leftrightarrow (ii)\Leftrightarrow
(iii)\Leftrightarrow (iv)\Leftrightarrow (v)$ and, similarly, we
also have the equivalences  $(i')\Leftrightarrow
(ii')\Leftrightarrow (iii')\Leftrightarrow (iv')\Leftrightarrow
(v')$.

Furthermore, for any given 2-functor $F:A\to C$, the application of
the hypothesis in $(iii)$ to the homotopy-fibre 2-functor
$F\!\downarrow\!-:C\to \underline{\mathbf{2}\Cat}$ just says that
$F$ has the property $\mathrm{B}_l$. Hence, $(iii)\Rightarrow (v')$.
Likewise, for any 2-functor $F:A\to C$, the hypothesis on
$-\!\downarrow\!F:C^{\mathrm{op}}\to \underline{\mathbf{2}\Cat}$ implies that
$F$ has the property $\mathrm{B}_r$, whence $(iii')\Rightarrow (v)$,
and the proof is complete.
\end{proof}

To finish, let us remark that the class of 2-categories satisfying
the conditions in Corollary \ref{corfin} above includes those
2-categories $C$ where, for each 1-cell $f:c\to c'$, there exists a
1-cell $f':c'\to c$ such that $[f'\circ f]=[1_c]\in \pi_0C(c,c)$ and
$[f\circ f']=[1_{c'}]\in \pi_0C(c',c')$.  In particular, the result
applies to 2-groupoids, whose 1-cells are all invertible, which,
recall, are equivalent to  crossed modules over groupoids by Brown
and Higgins's ~\cite{B-H1981}.

\bibliography{library}{}
\bibliographystyle{amsplain}

\end{document}